\documentclass[preprint, 11pt,a4paper]{elsarticle}
\usepackage{amsfonts}
\usepackage[numbers]{natbib}
\usepackage{enumerate}
\usepackage{nicefrac}
\usepackage{mathrsfs}
\usepackage[utf8]{inputenc}
\usepackage[english]{babel}
\usepackage{lineno}
\usepackage{caption}
\usepackage{amsmath,amssymb,amscd} 
\usepackage{amsthm} 
\usepackage{cancel}
\usepackage{mathtools}
\usepackage{microtype}
\usepackage{comment}
\usepackage{mathbbol}
\usepackage{tikz}
\usetikzlibrary{decorations.markings,arrows.meta}

\newtheorem{theorem}{Theorem}[section]
\newtheorem{lemma}[theorem]{Lemma}
\newtheorem{prop}[theorem]{Proposition}
\newtheorem{defn}[theorem]{Definition}
\newtheorem{cor}[theorem]{Corollary}

\newtheorem{ex}[theorem]{Example}

\newcommand{\T}{\mathcal{T}}
\newcommand{\underlineT}{\underline{T}}
\newcommand{\overlineT}{\overline{T}}
\newcommand{\ulp}{\underline{p}}
\newcommand{\olp}{\overline{p}}
\newcommand{\X}{\mathcal{X}}
\newcommand{\Y}{\mathcal{Y}}
\newcommand{\SSS}{\mathcal{S}}
\newcommand{\N}{\mathbb{N}}
\newcommand{\R}{\mathbb{R}}

\newcommand{\one}{\mathbb{1}}
\newcommand{\PP}{\mathbb{P}}

\newcommand{\W}{\mathscr{W}}
\newcommand{\C}{\mathcal{C}}
\newcommand{\D}{\mathcal{D}}
\newcommand{\A}{\mathcal{A}}
\newcommand{\E}{\mathcal{E}}
\newcommand{\F}{\mathcal{F}}

\newcommand{\Pcal}{\mathcal{P}}
\newcommand{\EE}{\mathbb{E}}

\newcommand{\uppereach}{ \rightharpoonup }
\newcommand{\lowereach}{ \rightharpoondown 
  }
  \newcommand{\ch}{\operatorname{co}}
\newcommand{\notuppereach}{\not\rightharpoonup}
\newcommand{\notlowereach}{ \not\rightharpoondown
  }
\DeclareMathOperator*{\argmax}{argmax}

\usepackage{xcolor}
\usepackage[linesnumbered,ruled,vlined]{algorithm2e}
\usepackage{doi}

\journal{International Journal of Approximate Reasoning}

\begin{document}

\begin{frontmatter}

\title{Computing Lower and Upper Hitting Probabilities for Imprecise Markov Chains}

\author{Marco Sangalli\corref{cor1}\fnref{equal}}
\ead{m.sangalli@tue.nl}
\author{Rick Reubsaet\fnref{equal}}
\author{Erik Quaeghebeur}
\author{Thomas Krak}

\fntext[equal]{Equal Contribution}
\cortext[cor1]{Corresponding Author}

\affiliation{organization={Eindhoven University of Technology},
            city={Eindhoven},
            country={Netherlands}}

\begin{abstract}
We study the computation of lower and upper probabilities of hitting a target set of states for imprecise Markov chains, where transition uncertainty is modelled by a convex set of transition matrices.  
In the precise case, hitting probabilities are the minimal nonnegative solution of a linear system and admit a closed-form expression. 
We investigate the notion of reachability in the imprecise setting. The literature review highlights several different definitions of lower reachability; thus, we explore the relations among them and present examples to clarify their logical implications. 
Using this revised definition of reachability for imprecise Markov chains, we partition the state space into classes of states whose hitting probabilities are trivially zero or one, and those which require further computation. For these nontrivial states, we show that the lower hitting probability is the unique solution of a nonlinear fixed-point equation, while the same does not hold for upper hitting probabilities.
For the practical computation of lower and upper hitting probabilities, we propose iterative algorithms that alternate between solving a linear system and choosing an extreme point from the set of transition matrices. 
Numerical experiments demonstrate that, in practice, these algorithms converge in substantially fewer iterations than the theoretically established worst-case bound.
\end{abstract}


\end{frontmatter}

\section{Introduction}
Hitting probabilities play a central role in the analysis of stochastic systems, capturing the likelihood that a process enters a designated set of target states in some finite time. In the classical precise setting of homogeneous Markov chains, these probabilities arise as the minimal nonnegative solution of a linear system and admit a closed-form representation via the fundamental matrix \cite{norrisbook1997markov,kemeny1976finite}. Hitting probabilities can also be regarded as discrete harmonic functions and enjoy a deep analogy with node voltages in potential theory and electrical networks~\cite{DoyleSnell1984}. 

In practical applications there is often uncertainty about the transition dynamics, which can be estimated from limited data or elicited from domain experts, yielding a convex set of transition matrices rather than a single, exact one.
A natural and rigorous way to represent this kind of epistemic uncertainty is through credal sets -- convex sets of probability distributions -- leading to the framework of imprecise Markov chains (IMCs)~\cite{DeCooman2009, HermansSkulj2014, hartfiel_seneta_1994}. In this framework, each row of the transition matrix is specified by a convex set of admissible probability distributions, and the inference quantities of interest are lower and upper hitting probabilities. These bound the true probability of reaching the target states under all compatible transition matrices.
These values provide conservative guarantees that are crucial in high-stakes decision-making contexts.
They naturally connect to MDPs as they correspond to the optimal value in reachability problems~\cite{Puterman1994, BaierKatoen2008}, and to stochastic-game formulations as they correspond to value bounds in zero-sum reachability games~\cite{Bovy2024}.

Our contributions build on the work of Krak et al.~\cite{krak2019hitting, krak2021comphit}. They developed characterisations of hitting times and probabilities for IMCs under various model semantics and presented algorithms for computing lower and upper hitting times. 
We extend this work by (i) discussing and clarifying the definition of reachability in the imprecise setting, (ii) giving a more general characterisation of lower and upper hitting probabilities, and (iii) proposing iterative algorithms to compute these quantities.

The structure of the paper is as follows.
Section \ref{sec:preli} introduces the preliminary notions needed for the following sections. In particular, we define hitting probabilities for homogeneous Markov chains, we explain the concept of imprecise Markov chains, and we introduce lower and upper transition operators. 
Section \ref{sec:charac} begins our contributions. There, we formalise the notions of lower and upper reachability in the imprecise setting; we review the different definitions that appear in the literature, clarify the logical relationships between them, and give a few illustrative examples. 
With these definitions in hand, we isolate those states whose (lower/upper) hitting probabilities are trivial, either zero or one.
We then prove that for nontrivial states, the lower hitting probability is the unique solution of a nonlinear fixed-point equation. Krak et al.~\cite{krak2019hitting} characterise the upper hitting probability as the minimal nonnegative solution of a nonlinear fixed-point equation, but this solution need not be unique. We illustrate this phenomenon with a counterexample and then show that uniqueness is restored under stronger assumptions on the set $\T$.
In Section \ref{sec:computing}, we turn our attention to developing iterative algorithms to compute lower and upper hitting probabilities. Inspired by the work of Krak~\cite{krak2021comphit} on hitting times, we propose iterative algorithms to compute lower and upper hitting probabilities that alternate between solving a linear system (restricted to a suitable subspace) and selecting an extreme point transition matrix from the credal sets. We rigorously prove their convergence and, in Section~\ref{sec:practical}, we provide a detailed analysis of their computational cost.
The empirical performance of these algorithms is explored in Section~\ref{sec:numanal}, which is dedicated to numerical experiments. Here we show how the theoretical bound on the number of iterations previously established is never attained in practice, and we point out an interesting relation between the density of the graph and the average number of iterations needed for convergence. We conclude with some summarising remarks and pointers to possible future work in Section~\ref{sec:conclusion}.

\section{Preliminaries}\label{sec:preli}
This section introduces the basic concepts of stochastic processes~\cite{Dudley2002}, Markov chains~\cite{norrisbook1997markov, kemeny1976finite}, hitting probabilities for homogeneous Markov chains~\cite{norrisbook1997markov}, and imprecise Markov chains~\cite{DeCooman2009, HermansSkulj2014, hartfiel_seneta_1994}.
\subsection{Stochastic Processes and Markov Chains}
Let us denote by $\N$ the set of positive integers and define $\N_0\coloneqq\N\cup\{0\}$. Let $\X$ be a discrete nonempty finite set of cardinality $N\in \N$ and denote by $2^\X$ its power set. Fix an arbitrary ordering over $\X$: this allows us to identify functions in $\R^\X$ and $N$-dimensional vectors.

Define the sample space $\Omega\coloneqq\X^{\N_0}$ and endow it with the product $\sigma$-algebra $\F\coloneqq(2^\X)^{\otimes\N_0}$, i.e. the $\sigma$-algebra generated by cylinder sets.
A \textit{stochastic process} $(X_n)_{n\in \N_0}$ is a family of random variables on this measurable space \[X_n:(\Omega, \F) \to (\X,2^\X)\] where we set $X_n(\omega)=\omega(n)$ for all $\omega \in \Omega$ and all $n\in \N_0$~\cite{Dudley2002}.
Let $\PP$ be a probability measure on the measurable space $(\Omega, \F)$, so that $(\Omega,\F,\PP)$ is a probability space.
The probability measure $\PP$ determines the law of the stochastic process; in the sequel, we identify the process with the probability measure $\PP$, which allows us to consider different processes $\PP$ over the same measurable space $(\Omega,\F)$~\cite{krak2019hitting}.

A stochastic process $(X_n)$ is said to be a \textit{Markov chain} if it satisfies the Markov property:
\begin{equation}\label{eq:markovprop}
    \PP(X_{n+1}=x_{n+1}\mid X_{0:n}=x_{0:n})=\PP (X_{n+1}=x_{n+1}\mid X_{n}=x_{n})
\end{equation}
for all $x_0, \dots, x_n,x_{n+1} \in \X$ and all $n\in\N_0$. This property is often interpreted as the process being memoryless; the transition probabilities depend only on the current state, and any previous history is not relevant.
A Markov chain $(X_n)$ is said to be (time-)\textit{homogeneous} if 
\begin{equation}
   \PP(X_{n+1}=y\mid X_n=x)=\PP(X_1=y\mid X_0=x) 
\end{equation}
for all $x,y\in \X$ and for all $n\in \N_0$.

A  matrix $T\in \R^{N\times N}$  is called a \textit{transition matrix} if it satisfies
\begin{enumerate}
    \item $T(x,y)\geq 0$ for all $x,y\in \X$, and
    \item $\displaystyle \sum_{y\in X}T(x,y)=1$ for all $x\in \X$.
\end{enumerate}
Note that each row $T(x,\cdot)$ of a transition matrix is a probability mass function on $\X$. Every transition matrix $T$ uniquely determines a homogeneous Markov chain $\PP_T$ up to its initial distribution, by setting
\begin{equation}
    \PP_T(X_{n+1}=y\mid X_n=x)=T{(x,y)}
\end{equation}
for all $x,y\in \X$ and all $n\in \N_0$. Thus, transition matrices are often used as the canonical parameter to specify homogeneous Markov chains. 
\begin{defn}[{{\cite[{Section 1.2}]{norrisbook1997markov}}}]\label{def:reach}
    Let \(x,y\in \X\), $x\ne y$, and let \(T\) be a transition matrix. We say that $x$ reaches $y$, denoted by  ``$x \to^{T} y$'', if 
    \[
    \PP_T(X_n=y \text{ for some $n\in \N_0$}\mid X_0=x)>0.
    \]
\end{defn}
The following proposition gives the equivalence between three formulations of reachability for homogeneous Markov chains.
\begin{prop}[{\cite[{Theorem 1.2.1}]{norrisbook1997markov}}]
\label{prop:reachcond}
     Let \(\PP_T\) be a homogeneous Markov chain defined on the state space \(\X\) with transition matrix \(T\). Let \(x,y\in \X\). Then the following statements are equivalent:
\begin{enumerate}
    \item \(x\to^{T} y\),
    \item there exists $n\in \N$ and a sequence of states \(x=x_0,\ldots,x_n=y\) in \(\X\) such that \(T(x_k,x_{k+1})>0\) for all $0\le k\le n-1$, and \(x_k\neq x_{j}\) for all $0\le k<j\le n$,
    \item \(T^n(x,y)>0\) for some \(n\in\N_0\).
\end{enumerate}
\end{prop}

\subsection{Hitting Probabilities for Markov Chains}
In this section, we introduce hitting indicators and hitting probabilities, then we characterise hitting probabilities for homogeneous Markov chains.
\begin{defn}
    Let $A\subseteq \X$ be a nonempty target set of states. The \textit{hitting indicator} is defined as the function \(\phi_A  : \Omega \to \{0, 1\}\) given by
\[
\phi_A (\omega) \coloneqq \sup_{n\in \N_0}\one_A(\omega(n)), 
\]
where $\one_A$ is the indicator function for $A$.
\end{defn}
This quantity is a Boolean function that has a value of one if the path reaches \(A\) and zero otherwise. 

\begin{defn}
We define the \textit{hitting probability} of a process \(\PP\) starting in \(x \in \X\) as the expected value of the hitting indicator
\(
\EE_\PP[\phi_A  \mid X_0=x].
\)
\end{defn}
The hitting probability of a process \(\PP\) starting in \(x \in \X\) can also be written as $\PP (X_n \in A \text{ for some $n\in \N_0$} \mid X_0=x)$, 
i.e. the probability that \(A\) is reached in some number of steps. In the remainder of the paper, we collect these hitting probabilities in the real-valued function \(p_A^\PP : \X \to [0, 1]\) defined as
\[
p_A^\PP(x) \coloneqq \EE_\PP[\phi_A  \mid X_0=x ].
\]

Let \(\PP_T\) be a homogeneous Markov chain with transition matrix \(T\), and define $p^T\coloneqq p_A^{\PP_T}$.
Let us introduce the set of states that cannot reach $A$, namely
\begin{align*}
    \C_T\coloneqq&\{x\in A^c \mid \forall y\in A: x\not\to^T y \}\\
    =&\{x\in A^c \mid  \forall y\in A, \forall  n\in \N_0: T^n(x,y)=0 \}.
\end{align*}
We can characterise the hitting probabilities of a homogeneous Markov chain as the minimal nonnegative solution of a linear system of equations.
\begin{prop}[{\cite[{Theorem 1.3.2}]{norrisbook1997markov}}]\label{prop:hitprob}
    Let \(\PP_T\) be a homogeneous Markov chain with transition matrix \(T\).
    Then \(p^T(x) = 1\) for all \(x \in A\) and \(p^T(x) = 0\) if and only if \(x \in \C_T\). Moreover, \(p^T\) is the minimal nonnegative solution of the system
\begin{equation}\label{eq:hitprob}
    p^T = \one_A + \one_{A^c} \cdot Tp^T. 
\end{equation}
\begin{proof}
    The only claim not explicitly given in Norris~\cite[{Theorem 1.3.2}]{norrisbook1997markov} is that \(p^T(x) = 0\) if and only if \(x \in \C_T\). This follows directly from the definition of $\C_T$ and Proposition \ref{prop:reachcond}.
\end{proof}
\end{prop}  

\subsection{Restrictions and Extensions of Functions and Matrices to a Subspace} Let $\emptyset\subset\SSS\subseteq \Y\subseteq \X$ and $f\in \R^{\Y}$. 
The definitions below give a canonical way to view functions and operators on the smaller set $\SSS$ as functions and operators on the larger set $\Y$ (and vice-versa).
We define the restriction of $f$ to $\SSS$ as
\begin{equation*}
    f|_{\SSS}(x)\coloneqq f(x) \text{ for all $x\in \SSS$}.
\end{equation*}
For any $g\in \R^{\SSS}$ we define $g^{\uparrow \Y}\in \R^{\Y}$, the extension of $g$ to $\Y$, as
\begin{equation*}
    \big[g^{\uparrow \Y}\big](x)\coloneqq
    \begin{cases}
        g(x) &\quad \text{ if  $x\in \SSS$},\\
        0 &\quad \text{ if $x\in \Y\setminus\SSS$}.
    \end{cases}
\end{equation*}
This extension is the right inverse of the restriction, namely
\[
\big(g^{\uparrow \Y}\big)\big|_{\SSS}=g,
\]
for all $g\in \R^{\SSS}$.
Let $M:\R^{\Y}\to \R^{\Y}$ . Define the restriction of $M$ to $\SSS$ (or, rather, to $\R^{\SSS}$) as 
\begin{equation*}
    M\big|_{\SSS}  g\coloneqq \left(M\big(g^{\uparrow \Y}\big)\right)\Big|_{\SSS} \ \text{ for all $g\in \R^{\SSS}$}.
\end{equation*}
If $M$ is a linear mapping, i.e. a $|\Y|\times |\Y|$ matrix, then $M|_{\SSS}$ coincides with the $|\SSS|\times |\SSS|$ submatrix indexed by $\SSS$.
\subsection{Closed-Form Solution for Hitting Probabilities}
In this section, we derive a closed-form expression for hitting probabilities for homogeneous Markov chains.
Most of the material in this section appears in standard references \cite[e.g][]{norrisbook1997markov, kemeny1976finite,DoyleSnell1984, GrinsteadSnelll1997}, however we present it here in full detail, since the terminology and results are essential for the subsequent development in the imprecise setting.

By Proposition \ref{prop:hitprob}, hitting probabilities of a homogeneous Markov chain \(\PP_T\) with transition matrix \(T\) are trivial for states in \(A\cup \C_T\). The system of equations in \eqref{eq:hitprob}, as we restrict to \(A^c\setminus \C_T\), can be reduced to
\begin{equation}\label{eq:hitprobAc}
    p^T|_{A^c\setminus \C_T} = (T p^T)|_{A^c\setminus \C_T}, 
\end{equation}
where \(p^T\) denotes the vector of hitting probabilities of \(\PP_T\). The following lemma, whose proof can be found in the appendix, allows us to rewrite the right-hand side of \eqref{eq:hitprobAc} in a more convenient way.

\begin{lemma}\label{lemma:rewritehitprob}
Let $\SSS\subset A^c\subset \X$, let \(T\) be a transition matrix, and let \(f\in\R^\X\) be any function such that \(f(x)=1\) for all \(x\in A\) and \(f(x)=0\) for all \(x\in \SSS\). Then
\begin{equation}
   (Tf)|_{A^c\setminus \SSS} = \bigl(T|_{A^c\setminus \SSS} (f|_{A^c\setminus \SSS})\bigr) + (T\one_A)|_{A^c\setminus \SSS}. 
\end{equation}
\end{lemma}
This lemma allows us to rewrite the system of equations in \eqref{eq:hitprobAc} as
\begin{align*}
    &p^T|_{A^c\setminus \C_T} = \bigl(T|_{A^c\setminus \C_T} p^T|_{A^c\setminus \C_T}\bigr) + (T\one_A)|_{A^c\setminus \C_T}\\
    &\ \Rightarrow \ (I-T|_{A^c\setminus \C_T})p^T|_{A^c\setminus \C_T}=(T\one_A)|_{A^c\setminus \C_T}.
\end{align*}
Hence, if we can show that the matrix \((I-T|_{A^c\setminus \C_T})\) is invertible, we can characterise \(p^T|_{A^c\setminus \C_T}\) and have a closed-form expression to compute it. 

We proceed with the following lemma, whose proof can be found in the appendix.
\begin{lemma}\label{lemma:412}
Let \(T\) be a transition matrix. Fix any \(x\in A^c\setminus \C_T\) and \(n\in\N\) with \(n\ge 2\). If \([T^n\one_A](x)>0\) and \([T^m\one_A](x)=0\) for all \(m<n\), then there exists some \(y\in A^c\setminus \C_T\), $y\ne x$ such that \(T(x,y)>0\) and \([T^{n-1}\one_A](y)>0\).
\end{lemma}

Let $\mathbf{0}$ and $\mathbf{1}$ be the vectors made of all zeros and ones\footnote{The size of these vectors will be clear from the context.}, respectively, and let $\SSS\subseteq \X$.
Then, the quantity $[(T|_{\SSS})^k\mathbf{1}](x)$, for \(x\in \SSS\) and \(k\in\N_0\), is equal to
\begin{align*}
[(T|_{\SSS})^k\mathbf{1}](x)&= \sum_{y\in \SSS} (T|_{\SSS})^k(x,y)\\
&=\sum_{y\in \SSS} \sum_{z_1,\dots, z_{k-1}\in \SSS} T|_{\SSS}(x,z_1)\cdots T|_{S}(z_{k-1},y)\\
&=\PP_T(X_1\in \SSS,\dots, X_k\in \SSS \mid X_0=x),
\end{align*}
i.e. the probability that a homogeneous Markov chain with transition matrix \(T\) starting in state \(x\) stays in \(\SSS\) after \(k\) steps. Hence the following lemma holds.

\begin{lemma}\label{lemma:413}
Let \(T\) be a transition matrix and let \(\SSS\subseteq \X\). Let \(f, g\in\R^{\SSS}\) be such that \(f\le g\). Then \((T|_\SSS)^k f \le (T|_\SSS)^k g\) for all \(k\in\N_0\). Moreover, we have that
\begin{equation}
    \mathbf{0}\le (T|_\SSS)^k \mathbf{1}\le \mathbf{1} \label{pizza1}
\end{equation}
for all \(k\in\N_0\).
\end{lemma}
The proof of this lemma can be found in the appendix.

Next, we show that for the choice $\SSS = A^c\setminus\C_T$, the upper bound in \eqref{pizza1} becomes strict for all sufficiently large $k\in\mathbb{N}$. 
Indeed, by definition, every state in $A^c\setminus\C_T$ can reach $A$ in finitely many steps.
Hence, if $k$ exceeds the minimal number of steps $n$ required for $x$ to reach a state outside $A^c\setminus\C_T$, the probability of remaining in $A^c\setminus\C_T$ after $k\ge n$ steps is strictly lower than $1$. The result below formalises this; its proof can be found in the appendix.
\begin{lemma}\label{lemma:414}
Let \(T\) be a transition matrix and let $A$ be the target set. There exists \(n\in\N\) such that for all \(k\in\N\) with \(k\ge n\), it holds that
\[
[(T|_{A^c\setminus \C_T})^k\mathbf{1}](x) < 1 \quad \text{for all } x\in A^c\setminus \C_T.
\]
\end{lemma}

Since the intuition for Lemma \ref{lemma:414} was that it holds because every state in \(A^c\setminus \C_T\) reaches a state outside of \(A^c\setminus \C_T\), the result actually holds for every subset of \(A^c\setminus \C_T\) as well. This slightly more general result will prove useful in the imprecise setting later on.
\begin{cor}\label{cor:415}
Let \(T\) be a transition matrix and let \(\SSS\subset \X\) be such that \(\C_T\subseteq \SSS\subset A^c\). Then there exists \(n\in\N\) such that for all \(k\in\N\) with \(k\ge n\), it holds that
\[
[(T|_{A^c\setminus \SSS})^k\mathbf{1}](x) < 1 \quad \text{for all } x\in A^c\setminus \SSS.
\]
\end{cor}
The proof of this corollary can be found in the appendix.

Using this result, we can conclude that powers of the transition matrix \(T\) restricted to \(A^c\setminus \SSS\) are smaller than one in norm.
\begin{cor}\label{cor:416}
Let \(T\) be a transition matrix and let \(\SSS\subset \X\) be such that \(\C_T\subseteq \SSS\subset A^c\). Then there exists some \(n\in\N\) such that
\[
||(T|_{A^c\setminus \SSS})^n|| < 1,
\]
where $||\cdot ||$ is the operator norm induced by the sup norm:
\begin{equation*}
    ||T||\coloneqq\sup_{\substack{f\in \R^N \\||f||_{\infty}=1 }} ||Tf||_{\infty}, \quad\text{ for $T\in \R^{N\times N}$.}
\end{equation*}
\end{cor}
The proof of this result can be found in the appendix.

Corollary \ref{cor:416} allows us to finally obtain the results we are after regarding the invertibility of \((I-T|_{A^c\setminus \C_T})\).
\begin{lemma}\label{lemma:417}
Let \(T\) be a transition matrix and let \(\SSS\subset \X\) be such that \(\C_T\subseteq \SSS\subset A^c\). Then 
\[
\lim_{k\to+\infty}(T|_{A^c\setminus \SSS})^k \mathbf{1} = \mathbf{0}
\]
and the inverse $(I-T|_{A^c\setminus \SSS})^{-1}$ exists and satisfies
\[
(I-T|_{A^c\setminus \SSS})^{-1} = \sum_{k=0}^\infty (T|_{A^c\setminus \SSS})^k.
\]
\end{lemma}

\begin{proof}
By Corollary \ref{cor:416}, there exists some \(n\in\N\) such that \(||(T|_{A^c\setminus \SSS})^n|| < 1\). This means that the restriction of $T$ to $A^c\setminus \SSS$ eventually becomes a contraction and the probability of remaining there decays to zero. 
Hence, the Neumann series $\sum_{k=0}^\infty (T|_{A^c\setminus \SSS})^k$ converges and the result follows~(Kemeny and Snell~\cite[{Theorem 1.11.1}]{kemeny1976finite}).
\end{proof}

The next proposition gives a closed-form expression for the vector of hitting probabilities of $\PP_T$.
\begin{prop}\label{prop:418}
Let \(\PP_T\) be a homogeneous Markov chain  with transition matrix \(T\). 
Then \(p^T(x)=1\) for all \(x\in A\) and \(p^T(x)=0\) if and only if \(x\in \C_T\). The vector \(p^T|_{A^c\setminus \C_T}\) is given by
\begin{equation}
 p^T|_{A^c\setminus \C_T} = \bigl(I-T|_{A^c\setminus \C_T}\bigr)^{-1}(T\one_A)|_{A^c\setminus \C_T},   
\end{equation}
and is the unique solution to 
\begin{equation}
    p^T = \one_A + \one_{A^c\setminus \C_T} \cdot Tp^T. \label{intel1}
\end{equation}
\end{prop}

\begin{proof}
The first two claims directly follow from Proposition \ref{prop:hitprob}. Restricting the system in Equation \eqref{eq:hitprob} to \(A^c\setminus \C_T\) yields
\[
p^T|_{A^c\setminus \C_T} = (T p^T)|_{A^c\setminus \C_T}.
\]
By Lemma \ref{lemma:rewritehitprob}, we find that
\[
p^T|_{A^c\setminus \C_T} = \bigl(T|_{A^c\setminus \C_T} p^T|_{A^c\setminus \C_T}\bigr) + (T\one_A)|_{A^c\setminus \C_T}.
\]
Hence,
\[
p^T|_{A^c\setminus \C_T} = \bigl(I-T|_{A^c\setminus \C_T}\bigr)^{-1}(T\one_A)|_{A^c\setminus \C_T},
\]
where we applied Lemma \ref{lemma:417} to guarantee the existence of the inverse. It then follows that $p^T$ is the unique solution to \eqref{intel1}
\end{proof}
The last lemma we present in this section is required for results in the imprecise setting later on. 
Given a transition matrix $T$, consider the induced directed weighted graph containing as vertices the states of \(\X\) and directed edges with weight \(T(x,y)\) for all \(x,y\in\X\). Then, the following lemma establishes that, from any state \(x\in\X\) which can reach \(A\), there must be a (directed) simple path from \(x\) to \(A\) using only nonzero weight edges for which the hitting probabilities of the states along the way are nondecreasing as we approach \(A\).
Essentially, this says that from every state \(x\in A^c\setminus \C_T\), there must be a way to “get closer” (i.e. increase the hitting probability) to \(A\). 

\begin{lemma}\label{lemma:419}
Consider a homogeneous Markov chain \(\PP_T\) with transition matrix \(T\) and let \(p^T \) be its vector of hitting probabilities. For all $x\in A^c\setminus \C_T$,
there exists a simple path
\(
x = x_0, x_1, \dots, x_n = z\in A
\)
with \(T(x_{k},x_{k+1})>0\) and 
\begin{equation}
   p^T(x_{k}) \le p^T(x_{k+1}) \quad \text{for all } k\in\{0,\dots,n-1\}. 
\end{equation}
\end{lemma}
The proof of this lemma can be found in the appendix.

\subsection{Imprecise Markov Chains}
Instead of a single stochastic process \(\PP\), we now consider a set of stochastic processes \(\Pcal\). If these processes are Markov chains, it is said to be an \textit{imprecise Markov chain} (IMC)~\cite{DeCooman2009, HermansSkulj2014, hartfiel_seneta_1994}. An interpretation may be that we do not know precisely which \(\PP \in \Pcal\) is the true process so that the set \(\Pcal\) serves as an “error margin” and, we hope, contains the true process \(\PP\) being modelled as well. Inferences with respect to the imprecise Markov chain \(\Pcal\) are defined using lower and upper expectations, which are defined as
\[
\underline{\EE}_\Pcal[\cdot\mid\cdot] \coloneqq \inf_{\PP\in \Pcal} \EE_\PP[\cdot\mid\cdot] \quad \text{and} \quad \overline{\EE}_\Pcal[\cdot\mid\cdot] \coloneqq \sup_{\PP\in \Pcal} \EE_\PP[\cdot\mid\cdot].
\]
These inferences are conservative, tight lower and upper bounds on inferences with respect to all stochastic processes in the set \(\Pcal\). Lower and upper probabilities, $\underline{\PP}_\Pcal$ and $\overline{\PP}_\Pcal$ respectively, are a special case of lower and upper expectations:
\[
\underline{\PP}_\Pcal(E\mid\cdot) \coloneqq \underline{\EE}_\Pcal[\one_E\mid\cdot] \quad \text{and} \quad \overline{\PP}_\Pcal(E\mid\cdot) \coloneqq \overline{\EE}_\Pcal[\one_E\mid\cdot],
\]
for all events $E\in \F$.

The set \(\Pcal\) representing the IMC will be parametrised through a nonempty set of transition matrices \(\T\). In this paper, we focus on (time-)homogeneous Markov chains  $\PP_T\in \Pcal$ that have their transition matrix $T\in \T$. This choice is justified by Krak et al.~\cite[{Theorem 18}]{krak2019hitting} showing that lower (and upper) hitting probabilities remain unchanged even when the set $\Pcal$ is enlarged to include (time-)inhomogeneous Markov processes compatible with $\T$.

Given a target set $A\subset \X$, we define the \textit{lower and upper hitting probabilities} of the imprecise Markov chain parametrised by $\T$ as 
\[
\ulp^\T \coloneqq \inf_{T\in \T} p^T \quad \text{and} \quad \olp^\T \coloneqq \sup_{T\in \T} p^T,
\]
where the infimum and supremum are taken pointwise.
Thus, we have  
\[
\ulp^\T(x) =  \underline{\PP}_{\Pcal}\Big(\exists n\in \N: X_n \in A \mid X_0 = x\Big),
\]
and
\[
\olp^\T(x) = \overline{\PP}_{\Pcal}\Big(\exists n\in \N: X_n \in A\mid X_0 = x\Big).
\]

Throughout the paper, we assume that the set $\T$ is nonempty, compact, convex, and with separately specified rows (SSR)~\cite{HermansSkulj2014, krak2019hitting}. This last property means that $\T$ can be represented as $N=|\X|$ compact and convex sets of probability distributions $\{\T_x\}_{x\in\X}$, one for each state $x\in\X$, so that $\T$ is the Cartesian product of these sets. Each of these sets is often called \textit{credal set}.
For such a set of transition matrices $\T$, we introduce valuable computational tools in its associated lower and upper transition operators.
\begin{defn}[\cite{krak2019hitting}] \label{defn:lowoper42}
    Given a set of transition matrices \(\T\) on $\X$, the corresponding lower and upper transition operators \(\underlineT,\overlineT: \R^\X \to \R^\X\) are defined as
\[
[\underlineT f](x) \coloneqq \inf_{T\in\T} [Tf](x) \quad \text{and} \quad [\overlineT f](x) \coloneqq \sup_{T\in\T} [Tf](x),
\]
for all \(f\in \R^{\X}\) and all \(x\in \X\).
\end{defn} 
From this definition, it follows that 
\[
[\underlineT f](x) = \underline{\EE}_{\Pcal_{\T}}[f(X_{n+1})\mid X_n=x] \quad \text{and} \quad [\overlineT f](x) = \overline{\EE}_{\Pcal_{\T}}[f(X_{n+1})\mid X_n=x],
\]
where $\Pcal_\T$ is the set of all homogeneous Markov chains with transition matrix in $\T$.
The following proposition establishes that these infima and suprema are always attained.
\begin{prop}[{\cite[{Lemma 1}]{krak2019hitting}}]
\label{prop:225}
    Let $\T$ be a set of transition matrices on $\X$. Then, for all $f\in \R^{\X}$,  there exists an extreme point \(T\in\T\) such that \(Tf =\underlineT f\), and an extreme point \(T'\in\T\) such that \(T'f = \overlineT f\).
\end{prop}

Lower and upper transition operators are conjugate, meaning that the lower transition operator \(\underlineT\) induces a corresponding upper transition operator
\begin{equation} \label{eq: conjugancy}
    \overlineT(\cdot) = -\underlineT(- \ \cdot ),
\end{equation}
and vice-versa. We now list some important properties of lower and upper transition operators. 
\begin{lemma}[\cite{DeCooman2009,krak2019hitting, krak2021comphit}] \label{lemma:proptransop}
    Given any lower transition operator \(\underlineT\) and its conjugate upper transition operator \(\overlineT\), for all \(f, g \in \R^{\X}\) and all \(n \in \N\) it holds that
\begin{enumerate}
    \item \(f \le g \Rightarrow {\underlineT^n} f \le {\underlineT^n} g\) and \({\overlineT^n} f \le {\overlineT^n} g,\)
    \item \({\underlineT^n} f + {\underlineT^n} g \le {\underlineT^n} (f+g)\) and \(\overlineT^{n}(f+g) \le \overlineT^{n} f + \overlineT^{n} g,\)
    \item \(||{\underlineT^n} f - {\underlineT^n} g|| \le ||f-g||\) and \(||\overlineT^{n} f - \overlineT^{n} g|| \le ||f-g||\),
    \item \(\min\limits_{y\in \X} f(y) \le [\underlineT f](x)\) and \([\overlineT f](x) \le \max\limits_{y\in \X} f(y),\)
    \item \(\underlineT(\lambda f) = \lambda \underlineT f\) and \(\overlineT(\lambda f) = \lambda \overlineT f\) for all $\lambda \geq 0$.
\end{enumerate}
\end{lemma} 
The singleton set $\T=\{T\}$ for a certain transition matrix $T$ has $\underlineT=\overlineT=T$. Thus, all these properties derived for lower and upper transition operators also hold for transition matrices.

\section{Characterization of Lower and Upper Hitting Probabilities} \label{sec:charac}
In this section, we characterise lower and upper hitting probabilities as the unique solutions of nonlinear fixed-point equations. We first formalise the notion of reachability in the imprecise setting.
\subsection{Lower and Upper Reachability} \label{subsec:reachability}
Before characterising lower and upper hitting probabilities, we identify the states that can reach the target set with positive lower (resp. upper) probability.
Recall that in the precise setting, for a homogeneous Markov chain with transition matrix \(T\), we focused on the set
of states \(A^c \setminus \C_T\). This was done because Proposition \ref{prop:hitprob} directly yields the hitting probabilities for other states, with the hitting probability equal to one for states in \(A\) and equal to zero for states in \(\C_T\). Clearly, the result for the states in \(A\) carries over to the imprecise setting, as \(p^T(x)=1\) for all \(T\in \T\) and all \(x\in A\). Hence also \(\ulp^\T(x)=\olp^\T(x)=1\) for all \(x\in A\).
To characterise the states with zero lower and upper hitting probability, we need suitable definitions of reachability for the imprecise setting. 
\begin{defn} [Upper reachability $\operatorname{\mathbf{(UR0)}}$]\label{def:uppereach}
    Given a set of transition matrices $\T$ and a set of states $\D\subseteq \X$ we say that $x$ upper reaches $\D$ if there exists $n\in \N_0$ and $T\in \T$ such that $[T^n \one_\D](x)>0$.
We denote by $x\uppereach \D$ if $x$ upper reaches $\D$.
\end{defn}
In the literature~\cite{ DeCooman2009,hermans2012characterisation, sangalli2025meeting} another definition often appears, namely 
\begin{description}
    \item[(UR1):] $x$ upper reaches $\D$ if there exists $n\in \N_0$ such that $[\overlineT^n\one_{\D}](x)>0$.
\end{description}
We show that these two definitions are equivalent.
\begin{theorem}\label{th:equivalenceURdefns}
     Definition $\operatorname{(UR0)}$ and definition $\operatorname{(UR1)}$ are equivalent.
\end{theorem}
\begin{proof}
 $\operatorname{(UR0)}\Rightarrow \operatorname{(UR1)}:$ This implication follows directly from the fact that
\begin{align}
    [\overlineT^n\one_{\D}](x)\coloneqq\sup_{T_1,\dots, T_n\in \T} [T_1\cdots T_n \one_\D](x).
\end{align}
Thus, if there exist $n\in \N_0$ and $\hat{T}\in \T$ such that $[\hat{T}^n \one_\D](x)>0$ (definition (UR0)), then
\begin{align*}
    [\overlineT^n\one_{\D}](x)&=\sup_{T_1,\dots, T_n\in \T} [T_1\cdots T_n \one_\D](x)\\
    &\ge \sup_{T\in \T} [T^n\one_\D](x)\\
    &\ge [\hat{T}^n\one_\D](x) >0.
\end{align*}
$\operatorname{(UR1)}\Rightarrow \operatorname{(UR0)}:$ Now suppose that there exists $n\in \N_0$ such that $ [\overlineT^n\one_{\D}](x)>0$.
Since $\T$ is compact, this supremum is attained, i.e. there exist $T_1,\dots, T_n\in \T$ such that
\[
[T_1\cdots T_n\one_\D](x)=[\overlineT^n\one_{\D}](x)>0.
\]
From this, we get 
\[
[T_1\cdots T_n\one_\D](x) = \sum_{\substack{z_1,\dots, z_{n-1}\in \X \\ y\in \D}} T_1(x,z_1)\cdots T_n(z_{n-1},y)>0,
\]
hence there exists a sequence $(\hat{z}_k)_{k=1}^n\coloneqq\hat{z}_0,\hat{z}_1, \dots, \hat{z}_{n-1},\hat{z}_n$, with $x=\hat{z}_0$ and $y=\hat{z}_n$, such that 
\[
T_1(x,\hat{z}_1)\cdots T_n(\hat{z}_{n-1},y)>0.
\]
We now construct a matrix $\Tilde{T}$ row-wise using convexity and the SSR property. For all $z\in \X$ define
\begin{equation*}
\Tilde{T}(z,\cdot)\coloneqq
\begin{cases}
\frac{1}{|\{\hat{z}_k=z\}|}\sum_{\hat{z}_k=z} T_{k+1}(\hat{z}_{k},\hat{z}_{k+1})& \text{if $z\in (\hat{z}_k)_k$ }\\
S(z,\cdot) \text{ for an arbitrary $S\in \T$} & \text{if $ z\notin (\hat{z}_k)_k$}.
\end{cases}
\end{equation*}
By convexity and SSR, $\Tilde{T}\in \T$, and $\Tilde{T}(z,w)>0$ for every pair of consecutive states $z,w\in \X$ in the sequence $(\hat{z}_k)_k$.   
It follows that 
\begin{align*}
    [\Tilde{T}^n \one_{\D}](x) &= \sum_{\substack{z_1,\dots, z_{n-1}\in \X \\ y\in \D}} \Tilde{T}(x,z_1)\cdots \Tilde{T}(z_{n-1},y)\\
    &\ge \Tilde{T}(x,\hat{z}_1)\cdots \Tilde{T}(\hat{z}_{n-1},y)>0,
\end{align*}
which is what we wanted to prove.
\end{proof}

\begin{defn} [Lower reachability $\operatorname{\mathbf{(LR0)}}$] \label{def:lowereach}
    Given a set of transition matrices $\T$ and a set of states $\D\subseteq \X$ we say that $x$ {lower reaches} $\D$ if for all $T\in \T$ there exists $n\in \N_0$ such that $[T^n\one_\D](x)>0$. 
We denote by $x\lowereach \D$ if $x$ lower reaches $\D$.
\end{defn}
We present three alternative natural definitions of lower reachability that often appear in the literature~\cite{debock2025convergent, sangalli2025meeting, debock2017imprecise, erreygers2023}.
\begin{description}
    \item[(LR1):] $x$ lower reaches $\D$ if $x\in \D_{k^*}$ where $(\D_k)_{k\in \N_0}$ is a nondecreasing sequence given by $\D_0=\D$ and 
    \begin{equation*}
        \D_{k+1}=\D_k\cup \{z\in \X\setminus \D_k \mid [\underlineT\one_{\D_k}](z)>0\},
    \end{equation*}
    and where $k^*$ is the first index for which $\D_{k^*}=\D_{k^*+1}$.~\cite{erreygers2023, debock2017imprecise}
    \item[(LR2):] $x$ lower reaches $\D$ if there exists $n\in \N_0$ such that for all $T\in \T$ it holds that $[T^n\one_\D](x)>0$.
    \item[(LR3):] $x$ lower reaches $\D$ if $[\underlineT^n \one_\D](x)>0$ for some $n\in \N_0$.~\cite{debock2025convergent, sangalli2025meeting} 
\end{description}
One reason for this abundance of possible definitions of lower reachability is that the quantifiers do not commute, i.e. in general 
\begin{equation*}
    \text{``$\exists n\in \N_0:\forall T\in \T, \text{(property)} $'' and ``$\forall T\in \T,\exists n\in \N_0: \text{(property)}$''}
\end{equation*}
are not equivalent.
Conversely, for upper reachability the quantifiers ``$\exists n\in \N_0$'' and ``$\exists T\in \T$'' commute: each formulation just requires the existence of at least one suitable pair $(n,T)\in \N_0\times \T$.

\begin{theorem}
   Definition $\operatorname{(LR0)}$ and definition $\operatorname{(LR1)}$ are equivalent.
\end{theorem}
\begin{proof}
The set $\E\subseteq \X $ of states that lower reach $\D$ can be seen as
\begin{align*} 
 \E&\coloneqq \{x\in \X \mid x \lowereach \D \}\\
&\ = \{x\in \X \mid \forall T\in \T, \exists n\in \N_0: [T^n\one_{\D}](x)>0\}\\
&\ =\bigcup_{k\in \N_{0}} \{x\in \X \mid \forall T\in \T, \exists n\le k: [T^n\one_{\D}](x)>0\}=:\bigcup_{k\in \N_{0}}  \E_k.
\end{align*}
To prove the equivalence $\operatorname{(LR0)}\Leftrightarrow \operatorname{(LR1)}$ it suffices to show that $\D_k=\E_k$ for all $k\in \N_0$. 

We prove this by induction on $k$. The induction base $k=0$ is trivial since $\D_0=\D=\E_0$. Assuming that $\D_k=\E_k$ we show that $\D_{k+1}=\E_{k+1}$, and to do so we prove the two inclusions,  $\D_{k+1}\subseteq \E_{k+1}$ and $\E_{k+1}\subseteq \D_{k+1}$. 

First, we observe that $(\D_k)$ and $(\E_k)$ are sequences of nested sets, namely $\D_k\subseteq \D_{k+1}$ and $\E_k\subseteq \E_{k+1}$. To show that $\D_{k+1}\subseteq \E_{k+1}$, we note that for all $x\in \D_{k+1}\setminus \D_k$, we have $[\underlineT \one_{\D_k}](x)>0$. Therefore, for all $T\in \T$ there is a state $y=y(T)\in \D_k$ such that $T(x,y)>0$, and, since $\E_k=\D_k \ni y$, there exists $n\le k$ such that $[T^n\one_\D](y)>0$. It follows that
\[
[T^{n+1}\one_\D](x)\ge T(x,y)[T^n\one_\D](y)>0,
\]
thus $x\in \E_{k+1}$, proving that $\D_{k+1}\subseteq \E_{k+1}$. 

To show that $\E_{k+1}\subseteq \D_{k+1}$, it suffices to prove that if $x\notin \D_{k+1}$ then $x\notin \E_{k+1}$, namely there exists $\hat{T}\in \T$ such that $[\hat{T}^{n+1}\one_\D](x)=0$ for all $n\le k$. We construct the matrix $\hat{T}$ row-wise starting from $x$: since $x\notin \D_{k+1}$, there exists a matrix $S_x\in \T$ such that $[S_x\one_{\D_{k}}](x)=0$, and we set $\hat{T}(x,\cdot)= S_x(x,\cdot)$. Observe that
\begin{align*}
    [\hat{T}^{n+1}\one_\D](x)&=\sum_{y\in \D_k} \overset{=0}{\cancel{\hat{T}(x,y)}}[\hat{T}^n\one_\D](y)+ \sum_{z\notin \D_k} \hat{T}(x,z)[\hat{T}^n\one_\D](z)\\
    &=\sum_{z\notin \D_k} \hat{T}(x,z)[\hat{T}^n\one_\D](z),
\end{align*}
for all $n\le k$.
Now, for all $z\notin \D_{k}=\E_k$, there exists $S_z\in \T$ such that $[S_z^n\one_\D](z)=0$ for all $n\le k$. By setting $\hat{T}(z,\cdot)=S_z(z,\cdot)$, for all $z\notin \D_k$ we get 
\[
[\hat{T}^{n+1}\one_\D](x)=\sum_{z\notin \D_k} \hat{T}(x,z)\overset{=0}{\cancel{[\hat{T}^n\one_\D](z)}}=0,
\]
which is what we wanted to prove.
\end{proof}
In the rest of the paper we use (LR0) (Definition \ref{def:lowereach}) (or the equivalent (LR1)) to define lower reachability. (LR0) is weaker than (LR2) and (LR3), as stated in the following proposition.
\begin{prop}\label{prop:chainofequiv}
    We have the following chain of implications: 
    \begin{equation}
    \operatorname{(LR3)} \Rightarrow \operatorname{(LR2)} \Rightarrow \operatorname{(LR1)} \Leftrightarrow \operatorname{(LR0)}.
\end{equation}
\end{prop}
\begin{proof}
    The implication $ \text{(LR3)} \Rightarrow \text{(LR2)}$ holds because if $[\underlineT^n \one_\D](x)>0$ for some $n\in \N_0$ then for all sequences of transition matrices $T_1, \dots, T_n\in \T$ we have $[T_1\cdots T_n\one_\D](x)>0$. It follows that $[T^n\one_\D](x)>0$ for all $T\in \T$.

    Now assume (LR2): there exists $n^*\in\mathbb N_0$ such that $[T^{n^*}\one_\D](x)>0$ for all $T\in\T$. Since (LR0) requires for all $T\in \T$ the existence of a $n=n(T)$ with the same property, the universal $n^*$ satisfies (LR0).
\end{proof}
The following examples show why these implications cannot be reversed.
\begin{ex}
    This example shows that $\operatorname{(LR0)}\not\Rightarrow \operatorname{(LR2)}$.
    Let $\X=\{1,2,3,4\}$ be the state space, and let $\{3\}$ be the target set. Let the set of transition matrices on $\X$ be
    \begin{equation*}
        \T=\left\{T=
        \begin{bmatrix}
            0 & p & 1-p&0\\
            0 & 0 & 1&0\\
            0 & 0 & 0&1\\
            0& 0& 0& 1
        \end{bmatrix}
        : p\in [0,1]
        \right\}.
    \end{equation*}
    The situation is represented by the following graph.
    \begin{center}
    \begin{tikzpicture}[
        xscale=1.8,yscale=1.8,
        every node/.style={draw=black,circle,fill=black!70,inner sep=2pt},
        redNode/.style={draw=red, fill=red!50},
        blueNode/.style={draw=blue, fill=blue!55},
        blackNode/.style={draw=black, fill=black},
        every label/.style={rectangle,fill=none,draw=none},
        every edge/.style={draw,bend right,looseness=0.3,
            postaction={
                decorate,
                decoration={
                    markings,
                    mark=at position 0.8 with {\arrow[#1]{Stealth}}
                }
            }
        }
    ]
        \node[blueNode,label=180:\textbf{1}] (1) at (2,55) {};
        \node[blueNode,label=90:\textbf{2}] (2) at (2.8,55.6) {};
       \node[blueNode,label=45:\textbf{3}] (3) at (3.6,55) {};
       \node[blueNode,label=90:\textbf{4}] (4) at (4.6,55) {};
        \path (1) edge[bend left=0, dashed] node[above, draw=none, fill=none, xshift=-2pt] {$p$} (2);
        \path (2) edge[bend left=0] (3);
        \path (1) edge[bend right, dashed] node[below,draw=none, fill=none, yshift=7pt] {$1-p$} (3);
         \path (3) edge[bend left=0] (4);
         \path (4) edge[relative=false,out=270,in=345,min distance=3.5ex] (4);
    \end{tikzpicture}
    \end{center}
    For all $T\in \T$ there exists $n\in \N$ (either 1 or 2) such that $T^n(1,3)>0$. Hence, $1$ lower reaches $3$, under $\operatorname{(LR0)}$. However, it is not true that there exists $n\in \N$ such that $T^n(1,3)>0$ for all $T\in \T$. In fact, the matrix $T_1\in \T$, defined by fixing $T_1(1,2)=p=1$, has $T_1^1(1,3)=0$, $T_1^2(1,3)=1$ and $T^n_1(1,3)=0$ for all $n\ge 3$. On the other hand, the matrix $T_0\in \T$, defined by fixing $T_0(1,2)=p=0$, has $T_0^1(1,3)=1$ and $T_0^n(1,3)=0$ for all $n\ge 2$. So $1$ does not lower reach $3$, under $\operatorname{(LR2)}$, because for all $n\in \N_0$ there exists $T\in \T$ such that $[T^n\one_{\{3\}}](1)=0$.
\end{ex}
\begin{ex}\label{ex:crisis}
    This example shows that $\operatorname{(LR2)}\not\Rightarrow \operatorname{(LR3)}$.
    Let $\X=\{1,2,3\}$ be the state space, and let $\{3\}$ be the target set. Let the set of transition matrices on $\X$ be
    \begin{equation*}
        \T=\left\{T=
        \begin{bmatrix}
            0 & 1-q & q\\
            0 & 0 & 1\\
            1 & 0 & 0\\
        \end{bmatrix}
        : q\in [0,1]
        \right\}.
    \end{equation*}
    The situation is represented by the following graph:
    \begin{center}
    \begin{tikzpicture}[
        xscale=1.8,yscale=1.8,
        every node/.style={draw=black,circle,fill=black!70,inner sep=2pt},
        redNode/.style={draw=red, fill=red!50},
        blueNode/.style={draw=blue, fill=blue!55},
        blackNode/.style={draw=black, fill=black},
        every label/.style={rectangle,fill=none,draw=none},
        every edge/.style={draw,bend right,looseness=0.3,
            postaction={
                decorate,
                decoration={
                    markings,
                    mark=at position 0.8 with {\arrow[#1]{Stealth}}
                }
            }
        }
    ]
        \node[blueNode,label=270:\textbf{1}] (1) at (2,54.4) {};
        \node[blueNode,label=0:\textbf{2}] (2) at (2.7,55.6) {};
       \node[blueNode,label=180:\textbf{3}] (3) at (1.3,55.6) {};
        \path (1) edge[bend left=0, dashed] node[right, draw=none, fill=none, xshift=4pt] {$1-q$} (2);
        \path (2) edge[bend left=0] (3);
        \path (1) edge[bend right, dashed] node[right,draw=none, fill=none, yshift=2.6pt] {$q$} (3);
        \path (3) edge[bend right] (1);
    \end{tikzpicture}
    \end{center}
    If $n=5$, all $T\in \T$ are such that $T^5(1,3)>0$, i.e. in five steps, whichever $T$ is selected, the process has a positive probability of being in state 3, if it has started in state 1. Therefore, 1 lower reaches 3, under $\operatorname{(LR2)}$.
    However, for all $n\in \N$ we have 
    \begin{equation*}
        [\underlineT^n \one_{\{3\}}](1)=0,
    \end{equation*}
    implying that  1 does not lower reach 3, under $\operatorname{(LR3)}$.
    This comes from the fact that, since we are applying the lower transition operator $\underlineT$ $n$ times, there is no guarantee that the same transition matrix achieves the minimum.
    Given $n\in \N$, we can find a sequence of $n$ transition matrices that minimises the probability of hitting state 3 in exactly $n$ steps, and we find that this probability is zero. 
   Given $n\in \N$, the sequence of transition matrices $(T_k)_{k=1}^n$ defined as
    \begin{equation*}
    T_k\coloneqq
        \begin{cases}
            S_0, & \text{if $k=n$},\\
            S_1, & \text{otherwise},
        \end{cases}
    \end{equation*}
    where $S_{0,1}\in \T$ are the two matrix defined by setting either $q=0$ or $q=1$, ensures to avoid state 3 at time $n$. Thus, 1 does not lower reach 3, under $\operatorname{(LR3)}$.
    To check this, one may also observe that $\underlineT\one_{\{3\}}=\one_{\{2\}}$ (the matrix $S_0$ achieves this minimum), $\underlineT^2\one_{\{3\}}=\underlineT\one_{\{2\}}=\mathbf{0}$ (the sequence $(S_1, S_0)$ achieves this minimum), thus 
    \[
    [\underlineT^{n+2}\one_{\{3\}}](1)=[\underlineT^{n+1}\one_{\{2\}}](1)=[\underlineT^{n}\mathbf{0}](1)=0,
    \]
     for all $n\in \N$.
\end{ex}
There are particular cases in which definitions (LR0), (LR1), (LR2) and (LR3) are all equivalent. 
We say that a set $\D$ is \textit{closed} when for all $x\in \D, y\notin \D$ and for all $T\in \T$, we have $T(x,y)=0$. 
\begin{prop}\label{prop:allequiv}
    Under the assumption that $\D$ is closed, definitions $\operatorname{(LR0)}$, $\operatorname{(LR1)}$, $\operatorname{(LR2)}$ and $\operatorname{(LR3)}$ are equivalent.
\end{prop}
\begin{proof}
We show that $\text{(LR1)} \Rightarrow \text{(LR3)}$. The thesis will follow from this and Proposition \ref{prop:chainofequiv}.
We need to prove that, if $x\lowereach y$ under (LR1), then there exists $n\in \N$ such that 
\begin{equation*}
    \sum_{x_1,\dots, x_{n-1}\in \X} T_1(x,x_1)\cdots T_n(x_{n-1}, y)>0,
\end{equation*}
for all $T_1,\dots, T_n \in \T$. Let us pick $n=k^*$ recalling that, by hypothesis, we have $x\in \D_{k^*}$. For all $T_1\in \T$, there exists $x_1\in \D_{k^*-1}$ such that $T_1(x,x_1)>0$. Iterating this argument, we get to $\D$ in at most $k^*$ steps. If, for some sequence of matrices, there is a positive probability of hitting $\D$ at time $k'<k^*$, then, due to the closedness of $\D$, the same would hold at time $k^*$. 
\end{proof}
An alternative proof of the equivalence between (LR1) and (LR3) under the closedness hypothesis can also be found in \cite{debock2025convergent}.

\subsection{Lower Hitting Probabilities}
Having discussed the definition of reachability extensively, we can now use this definition to characterise lower hitting probabilities, i.e. $\ulp^\T=\inf_{T\in \T}p^T$.
\begin{defn}
We define the set of states $\A_\T\subset \X$ that cannot lower reach the target as
\begin{equation}
    \A_\T  \coloneqq \{ x \in A^c \mid x \notlowereach A\}.
\end{equation}
\end{defn}
Using the definition of lower reachability, we deduce a straightforward relationship between the sets $\A_\T$ and $\C_T$.
\begin{lemma}\label{lemma:422}
It holds that
\[
\A_\T = \bigcup_{T\in \T} \C_T.
\]
\end{lemma}
\begin{proof}
    The set $\A_\T$ can be rewritten as
\begin{align*}
    \A_\T &= \{ x \in A^c \mid \exists T\in \T : \forall n\in \N, \ [T^n \one_A](x)=0 \}\\
    &=\{x\in A^c\mid \exists T\in \T: x\in \C_T\},
\end{align*}
and the claim follows.
\end{proof}
We may interpret \(\A_\T \) as the set of states whose lower probability of reaching \(A\) is zero, since there exists a transition matrix that disconnects $\A_\T $ from $A$. 
We are finally able to characterise lower hitting probabilities as the unique solution of a system of equations.
\begin{prop}\label{prop:423}
The vector of lower hitting probabilities \(\ulp^\T\) satisfies \(\ulp^\T(x)=1\) for all \(x\in A\) and  \(\ulp^\T(x)=0\) if and only if \(x\in \A_\T \). Moreover, $\ulp^\T$ is the unique solution of
\begin{equation} \label{eq:lowhitprob}
    \ulp^\T = \one_A + \one_{A^c\setminus \A_\T}\cdot \underlineT \ulp^\T.
\end{equation}
\end{prop}

\begin{proof}
Let us start with the claims about the structure of \(\ulp^\T\). We know from Krak et al.~\cite[Corollary 19]{krak2019hitting} that
\[
\ulp^\T = \one_A + \one_{A^c}\cdot \underlineT \ulp^\T
\]
has a minimal nonnegative solution \(\ulp^\T\in [0,1]^{\X}\). By the definition of the system, this solution satisfies \(\ulp^\T(x)=1\) for all \(x\in A\). The fact that \(\ulp^\T(x)=0\) if and only if \(x\in \A_\T \) follows from the definition of lower hitting probability, Lemma \ref{lemma:422}, and Proposition \ref{prop:hitprob}. Therefore, $\ulp^\T$ is the minimal nonnegative solution of \eqref{eq:lowhitprob}.

Now we prove the uniqueness of $\ulp^\T$. 
Let $q\in [0,1]^\X$ be a solution of \eqref{eq:lowhitprob}. By minimality, we know that $q\ge \ulp^\T$, therefore we need to show that $q\le \ulp^\T$. 
Since $q$ and $\ulp^\T$ already agree on $A$ and $\A_\T$, it suffices to show that \(q|_{A^c\setminus \A_\T }\le\ulp^\T|_{A^c\setminus \A_\T }\).

By Krak et al.~\cite[Theorem 18]{krak2019hitting}, there exists $T^*\in \T$ such that $p^{T^*}=\ulp^\T$. Since $p^{T^*}(x)>0$ if and only if $x\in \A_\T$, by Proposition \ref{prop:hitprob}, we have $\C_{T^*}=\A_\T$ and
\begin{equation*}
   \ulp^\T= p^{T^*}=\one_A+\one_{A^c\setminus \A_\T}\cdot T^*p^{T^*}.
\end{equation*}
From \eqref{eq:lowhitprob} applied to $q$, we obtain
\[
q|_{A^c\setminus \A_\T}=(\underlineT q)|_{A^c\setminus \A_\T}\le (T^*q)|_{A^c\setminus \A_\T}.
\]
Now we apply Lemma \ref{lemma:rewritehitprob} and we get
\[
q|_{A^c\setminus \A_\T}\le (T^*q)|_{A^c\setminus \A_\T}=T^*|_{A^c\setminus \A_\T}q|_{A^c\setminus \A_\T}+(T^*\one_A)|_{A^c\setminus \A_\T}.
\]
Recursively expanding this expression for \(q|_{A^c\setminus \A_\T }\) yields, after \(m\in \N\) expansions
 \begin{equation*}
        q|_{A^c\setminus {\A_{\T}}}\le(T^*|_{A^c\setminus {\A_{\T}}})^m q|_{A^c\setminus {\A_{\T}}}+\left[\sum_{k=0}^{m-1}(T^*|_{A^c\setminus {\A_{\T}}})^k\right](T^*\one_A)|_{A^c\setminus {\A_{\T}}}.
\end{equation*}
Since $\C_{T^*}={\A_{\T}}$, we let $m\to +\infty$ and we apply Lemma \ref{lemma:417}:
\begin{equation*}
     q|_{A^c\setminus {\A_{\T}}}\le(I-T^*|_{A^c\setminus {\A_{\T}}})^{-1}(T^*\one_A)|_{A^c\setminus {\A_{\T}}}=p^{T^*}|_{A^c\setminus {\A_{\T}}}=\ulp^\T|_{A^c\setminus {\A_{\T}}},
\end{equation*}
where the penultimate equality follows from Proposition \ref{prop:418}. Thus $q=\ulp^\T$, which is what we wanted to prove.
\end{proof}

\subsection{Upper Hitting Probabilities}
We now turn our attention to characterising upper hitting probabilities, i.e. $\olp^\T=\sup_{T\in \T}p^T$. 

\begin{defn}\label{def:defW}
We define the set of states \({\W_{\T}} \) that cannot upper reach the target  as
\begin{equation}
{\W_{\T}} \coloneqq \{ x\in A^c \mid x\notuppereach A\}.
\end{equation}
\end{defn}
Similarly to lower hitting probabilities, there is a straightforward connection between sets $\W_\T$ and $\C_T$.
\begin{lemma}\label{lemma:426}
It holds that
\[
{\W_{\T}} = \bigcap_{T\in \T} \C_T.
\]
\end{lemma}
\begin{proof}
The set $\W_\T$ can be rewritten as
\begin{align*}
    {\W_{\T}} &= \{ x\in A^c \mid \forall T\in\T, \  \forall n\in \N: [T^n\one_A](x)=0\}\\
    &=\{x\in A^c \mid \forall T\in \T: x\in \C_T\},
\end{align*}
and the claim follows.
\end{proof}
We may interpret $\W_\T$ as the set of states whose upper probability of reaching $A$ is zero, since every transition matrix disconnects $\W_\T$ to $A$. 

The analogue of Proposition \ref{prop:423} for upper hitting probabilities fails in general. Nevertheless, we find that $\olp^\T$ satisfies \(\olp^\T(x)=1\) for all \(x\in A\) and \(\olp^\T(x)=0\) if and only if \(x\in \W_\T \).
These properties follow immediately from Krak et al.~\cite[Corollary 19]{krak2019hitting}, the definition of upper hitting probability, and Lemma \ref{lemma:426}. 
We know from Krak et al.~\cite[Corollary 19]{krak2019hitting} that $\olp^\T$ is the minimal nonnegative solution of $\olp^\T = \one_A + \one_{A^c}\cdot \overlineT \olp^\T$, therefore it is also the minimal nonnegative solution of
\begin{equation} \label{eq:system_uprob}
    \olp^\T = \one_A + \one_{A^c\setminus \W_\T}\cdot \overlineT \olp^\T.
\end{equation}
Unlike the lower hitting probability, the vector $\olp^\T$ need not be the unique solution of \eqref{eq:system_uprob}. Below we present an example in which another vector $q>\overline{p}^\T$ also satisfies \eqref{eq:system_uprob}.
\begin{ex}
    Let $\X=\{1,2,3\}$ be the state space, and let $\{2\}$ be the target set. Let the set of transition matrices on $\X$ be
    \begin{equation*}
        \T=\ch\left\{I, T_n=
        \begin{bmatrix}
            1-\frac{2}{n}& \frac{1}{n}-\frac{1}{n^2} & \frac{1}{n}+\frac{1}{n^2}\\
            0 & 1 & 0\\
            0 & 0 & 1\\
        \end{bmatrix}
        : n\ge 2
        \right\},
    \end{equation*}
    where by ``co'' we denote the convex hull and by ``$I$'' the identity matrix. 
    The situation is represented by the following graph:
    \begin{center}
    \begin{tikzpicture}[
        xscale=1.75,yscale=1.75,
        every node/.style={draw=black,circle,fill=black!70,inner sep=2pt},
        redNode/.style={draw=red, fill=red!50},
        blueNode/.style={draw=blue, fill=blue!55},
        blackNode/.style={draw=black, fill=black},
        every label/.style={rectangle,fill=none,draw=none},
        every edge/.style={draw,bend right,looseness=0.3,
            postaction={
                decorate,
                decoration={
                    markings,
                    mark=at position 0.8 with {\arrow[#1]{Stealth}}
                }
            }
        }
    ]
        \node[blueNode,label=180:\textbf{1}] (1) at (2,54.4) {};
        \node[blueNode,label=180:\textbf{2}] (2) at (1.15,55.6) {};
       \node[blueNode,label=180:\textbf{3}] (3) at (2.85,55.6) {};
        \path (1) edge[bend left=0] node[left, draw=none, fill=none, yshift=-1pt] {$\nicefrac{1}{n}-\nicefrac{1}{n^2}$} (2);
        \path (1) edge[bend right=0] node[right,draw=none, fill=none, yshift=-1pt] {$\nicefrac{1}{n}+\nicefrac{1}{n^2}$} (3);
        \path (1) edge[relative=false,out=270,in=345,min distance=3.5ex] node[right,draw=none, fill=none, yshift=0.1pt] {$1-\nicefrac{2}{n}$} (1);
        \path (2) edge[relative=false,out=15,in=90,min distance=3.5ex] node[right,draw=none, fill=none, yshift=-0.1pt] {$1$} (2);
        \path (3) edge[relative=false,out=15,in=90,min distance=3.5ex] node[right,draw=none, fill=none, yshift=-0.1pt] {$1$} (3);
    \end{tikzpicture}
    \end{center}
    Since $A=\{2\}$ and $\W_\T=\{3\}$, we have that, for every $T\in \T$, $p^T(2)=1$ and $p^T(3)=0$.
    When starting from state ${1}$ and evolving under the transition matrix $T_n$, the hitting probability $p^{T_n}$ satisfies
    \begin{align*}
        p^{T_n}(1)=\left(1-\dfrac{2}{n}\right)p^{T_n}(1) +\left(\dfrac{1}{n}-\dfrac{1}{n^2}\right).
    \end{align*}
    Therefore
    \begin{equation*}
        p^{T_n}(1)=\dfrac{n}{2} \ \dfrac{n-1}{n^2}=\dfrac{n-1}{2n}\xrightarrow{n \to \infty} \dfrac{1}{2},
    \end{equation*}
    and $p^{T_n}(1)<1/2$ for all $n\ge 2$. Since no transition matrix in $\T$ yields a hitting probability greater or equal than $1/2$, we have that $\olp^\T(1)=1/2$. The vector $\olp^\T=(1/2,1,0)^\top$ satisfies \eqref{eq:system_uprob}:
    \begin{equation*}
        1/2=\olp^\T(1)=\sup_{T\in \T} \left[T\cdot
        \begin{pmatrix}
            1/2\\1\\0
        \end{pmatrix}
        \right](1)=\left[ I\cdot 
        \begin{pmatrix}
            1/2\\1\\0
        \end{pmatrix}\right](1)=1/2,
    \end{equation*}
    where the third equality follows from the identity matrix $I\in \T$ attaining the supremum: $[\overlineT  \olp^\T](1)=[I  \olp^\T](1)$.

    The vector $q=(1,1,0)^\top$ satisfies \eqref{eq:system_uprob} too. This follows from the fact that
    \begin{equation*}
        T_nq=
        \begin{bmatrix}
            1-\frac{2}{n}& \frac{1}{n}-\frac{1}{n^2} & \frac{1}{n}+\frac{1}{n^2}\\
            0 & 1 & 0\\
            0 & 0 & 1\\
        \end{bmatrix}
        \begin{pmatrix}
            1\\1\\0
        \end{pmatrix}
        =\begin{pmatrix}
            1-\frac{1}{n}-\frac{1}{n^2}\\1\\0
        \end{pmatrix}< \begin{pmatrix}
            1\\1\\0
        \end{pmatrix}=q,
    \end{equation*}
    and that $\overlineT q=Iq=q$. More generally, every vector $q_\alpha \coloneqq (\alpha,1,0)^\top$ with $1/2\le \alpha \le 1$ is a solution of \eqref{eq:system_uprob}.
\end{ex}
To restore the uniqueness of $\olp^\T$ we may impose additional hypotheses; for instance, it suffices that every transition matrix $T\in\T$ has the same set $\C_T$, or equivalently $\W_\T=\A_\T$.

\begin{prop}\label{prop:427}
If the set $\T$ is such that $\W_\T=\A_\T$, then the vector of upper hitting probabilities $\olp^\T$ is the unique solution of 
\begin{equation} \label{eq:system_uprob1}
    \olp^\T = \one_A + \one_{A^c\setminus \W_\T}\cdot \overlineT \olp^\T.
\end{equation}
\end{prop}
\begin{proof}
Let $q\in [0,1]^\X$ be a solution of \eqref{eq:system_uprob1}. 
By minimality, we know that $q\ge \olp^\T$, therefore we need to show $q\le \olp^\T$. 
Since $q$ and $\olp^\T$ already agree on $A$ and $\W_\T$ it suffices to show that $q|_{A^c\setminus\W_\T}\le \olp^\T|_{A^c\setminus\W_\T}$. 
    By Proposition \ref{prop:225}, there exists $S\in \T$ such that $\overlineT q = Sq$. It follows that
\(
    q = \one_A + \one_{A^c\setminus \W_\T}\cdot S q.
\)
From \eqref{eq:system_uprob1}, we obtain
\begin{equation*}
    \olp^\T|_{A^c\setminus\W_\T}=(\overlineT \olp^\T)|_{A^c\setminus\W_\T}\ge (S\olp^\T)|_{A^c\setminus\W_\T}.
\end{equation*}
Now we apply Lemma \ref{lemma:rewritehitprob} and we get
\begin{equation*}
    \olp^\T|_{A^c\setminus\W_\T}\ge (S\olp^\T)|_{A^c\setminus\W_\T} = S|_{A^c\setminus\W_\T} \olp^\T |_{A^c\setminus\W_\T} + (S\one_A)|_{A^c\setminus\W_\T}.
\end{equation*}
Recursively expanding this expression for \(\olp^\T|_{A^c\setminus \W_\T }\) yields, after \(m\in \N\) expansions
 \begin{equation*}
        \olp^\T|_{A^c\setminus {\W_{\T}}}\ge(S|_{A^c\setminus {\W_{\T}}})^m \  \olp^\T|_{A^c\setminus {\W_{\T}}}+\left[\sum_{k=0}^{m-1}(S|_{A^c\setminus {\W_{\T}}})^k\right](S\one_A)|_{A^c\setminus {\W_{\T}}}.
\end{equation*}
Since $\C_{S}={\W_{\T}}$, we let $m\to +\infty$ and we apply Lemma \ref{lemma:417}:
\begin{equation*}
     \olp^\T|_{A^c\setminus {\W_{\T}}}\ge(I-S|_{A^c\setminus {\W_{\T}}})^{-1}(S\one_A)|_{A^c\setminus {\W_{\T}}}=q|_{A^c\setminus {\W_{\T}}},
\end{equation*}
where the last equality follows from Proposition \ref{prop:418}. Thus $q=\olp^\T$, which is what we wanted to prove.
\end{proof}

\subsection{Further Discussions on the Definition of Reachability}
An alternative and very natural definition of lower reachability can be formulated in terms of hitting probabilities: 
\begin{description}
     \item[(LRHP):] $x$ lower reaches $A$ if the lower probability of hitting $A$ is positive, i.e. $\ulp^\T(x)>0$.
\end{description}
 Proposition~\ref{prop:423} shows that this definition is equivalent to $\operatorname{(LR0)}$: indeed, $x \in \A_\T$, i.e. $x \notlowereach A$, if and only if $\ulp^\T(x) = 0$. Taking the negation, we obtain that $x$ lower reaches $A$ under (LR0) if and only if $\ulp^\T(x) > 0$, thus definitions $\operatorname{(LR0)}$ and $\operatorname{(LRHP)}$ are equivalent.
 
In the same way, the definition of upper reachability can be rewritten in terms of positivity of the upper hitting probability. 
These definitions of reachability for IMCs mirror the behaviour in the precise setting. In fact, also for a (precise) homogeneous Markov chain we have that $x$ reaches $A$ if and only if its hitting probability is positive, i.e. $p^T(x)>0$. 

\section{Computing Hitting Probabilities for IMCs}\label{sec:computing}
Having established a characterisation of lower and upper hitting probabilities, this section presents practical algorithms for computing these quantities.
\subsection{Computing Lower Hitting Probabilities}\label{subsec:computing_lower}
We show an algorithmic result for computing \(\ulp^\T\). This result, which emulates Krak~\cite[Proposition 7]{krak2021comphit}, provides a practical algorithm for calculating \(\ulp^\T\) by iterating over extreme points of \(\T\). The only condition is to start with \(T_1\in \T\) such that the trivial hitting probabilities are already correct (i.e. \(p_1(x)=0\) for \(x\in \A_\T \) and \(p_1(x)=1\) for \(x\in A\)). Definition (LR1) provides a practical algorithm to identify the set $\A_\T$ and the transition matrix $T_1\in \T$. In particular, if we set $\D=A$ then $\D_{k^*}^c=\A_\T$. For all $z\in \A_\T=\D_{k^*}^c$ we have $[\underlineT \one_{\D_{k^*}}](z)=0$, therefore there exists a matrix $S_z\in \T$ such that $[S_z\one_{\D_{k^*}}](z)=0$. We construct the initial matrix $T_1\in \T$ by setting $T_1(z,\cdot)=S_z(z,\cdot)$, for all $z\in \A_\T$, and an arbitrary admissible row for any other state. It follows that \(\C_{T_1} = \A_\T \).

\begin{prop}\label{prop:424algolow}
Let \(T_1\in \T\) be any transition matrix such that \(\C_{T_1} = \A_\T \). For all \(n\in \N\), let \(p_n\) be the vector of hitting probabilities for $T_n$, i.e. the unique solution of the linear system
\[
p_n = \one_A + \one_{A^c\setminus \C_{T_n}}\cdot T_n p_n,
\]
and let \(T_{n+1}\) be an extreme point of \(\T\) such that
\[
T_{n+1} p_n =\underlineT p_n.
\]
Then, the sequence \((p_n)_{n\in \N}\) is nonincreasing and its limit \(p^*\coloneqq\lim_{n\to+\infty} p_n\) is the vector of lower hitting probabilities, i.e. the unique solution of the system 
\begin{equation}\label{eq:lowhitprob2}
    \ulp^\T=\one_A + \one_{A^c\setminus \A_\T}\cdot \underlineT  \ulp^\T.
\end{equation}
\end{prop}
We delay the proof of this proposition until after the following lemma.
\begin{lemma}\label{lemma:ctnat}
    For the sequence of transition matrices $(T_n)_{n\in \N}$ defined in Proposition \ref{prop:424algolow} it holds that
    \begin{equation}
        \C_{T_n}=\A_\T, \label{eq:skibidi}
    \end{equation}
    for all $n\in \N$.
\end{lemma}
\begin{proof}
The proof uses induction on \eqref{eq:skibidi}.
For the base case $n = 1$, we have $\C_{T_1} = \A_\T $ by the choice of $T_1 \in \T$. Assume now that $\C_{T_n} = \A_\T $ for some $n \in \N$. We show that also $\C_{T_{n+1}} = \A_\T $. 
First, we observe that $\C_{T_{n+1}}\subseteq \A_\T=\bigcup_{T\in \T}\C_T$, by Lemma \ref{lemma:422}. What remains to prove is $\A_\T\subseteq \C_{T_{n+1}}$.
Let $x \in \A_\T$. 
Since $p_n$ is the vector of hitting probabilities for $T_n$, by Proposition \ref{prop:hitprob} we have
\[
[T_n p_n](x) = p_n(x) = 0,
\]
where the last equality follows from the induction hypothesis $\C_{T_n}=\A_\T$.
By construction, we have 
\[
T_{n+1} p_n = \underlineT p_n \le T_n p_n.
\]
 Therefore,
\[
0 = [T_n p_n](x) \ge [T_{n+1} p_n](x) = \sum_{y\in\X} T_{n+1}(x,y) p_n(y)\ge 0.
\]
Since the right-hand side is nonnegative, we must have $T_{n+1}(x,y)=0$ for every $y\in\X$ with $p_n(y)>0$. By Proposition~\ref{prop:hitprob} this implies $T_{n+1}(x,y)=0$ for all $y\notin\C_{T_n}=\A_\T$.  
From the arbitrariness of $x\in \A_\T$ we have $T_{n+1}(x,y)=0$ for all $x\in \A_\T$ and $y\notin \A_\T$.  Since $A\subseteq\A_\T^c$, we have $x\not\to^{T_{n+1}}A$ for all $x\in \A_\T$. Thus $x\in\C_{T_{n+1}}$ and $\A_\T\subseteq\C_{T_{n+1}}$, concluding the proof.
\end{proof}
\begin{proof}[Proof of Proposition \ref{prop:424algolow}]
The sequence $(p_n)_{n\in\N}$ is well-defined since \[p_n = \one_A + \one_{A^c\setminus \C_{T_n}} \cdot T_n p_n\] has a unique solution (Proposition \ref{prop:418}) and the existence of the extreme point $T_{n+1}$ is guaranteed by Proposition \ref{prop:225}.

By Proposition \ref{prop:418}, we have that, for all $n \in \N$, $p_n(x) = 1$ for all $x \in A$ and $p_n(x) = 0$ for all $x \in \C_{T_n}=\A_\T$, therefore $(p_n)_{n\in\N}$ is trivially convergent on $ A\cup \A_\T $.

Next, we show that $(p_n)_{n\in\N}$ is nonincreasing and convergent on $A^c \setminus \A_\T $. 
Recall that 
\(
T_{n+1} p_n = \underlineT p_n \le T_n p_n
\)
for all $n\in \N$.
This yields
\begin{align*}
    p_n|_{A^c\setminus \A_\T } &= (T_n p_n)|_{A^c\setminus \A_\T }\\ 
    &\ge (T_{n+1} p_n)|_{A^c\setminus \A_\T } \\
    &= T_{n+1}|_{A^c\setminus \A_\T } p_n|_{A^c\setminus \A_\T } + (T_{n+1} \one_A)|_{A^c\setminus \A_\T },
\end{align*}
where we used Lemma \ref{lemma:rewritehitprob} and the fact that  $\A_\T = \C_{T_n}$ for all $n \in \N$. Repeatedly expanding this inequality $m \in \N$ times yields 
\[
p_n|_{A^c\setminus \A_\T } \ge (T_{n+1}|_{A^c\setminus \A_\T })^m p_n|_{A^c\setminus \A_\T } + \left[\sum_{k=0}^{m-1} (T_{n+1}|_{A^c\setminus \A_\T })^k \right](T_{n+1} \one_A)|_{A^c\setminus \A_\T }.
\]
Letting $m \to +\infty$ and applying Lemma \ref{lemma:417} yields 
\[
p_n|_{A^c\setminus \A_\T } \ge (I - T_{n+1}|_{A^c\setminus \A_\T })^{-1} (T_{n+1} \one_A)|_{A^c\setminus \A_\T } = p_{n+1}|_{A^c\setminus \A_\T },
\]
where the last step follows from Proposition \ref{prop:418}. Thus the sequence $(p_n)_{n\in\N}$ is nonincreasing.
Since $(p_n)_{n\in\N}$ is nonincreasing and $p_n \in [0,1]^{{\X}}$ for all $n \in \N$, it follows that the limit 
\(
p^* \coloneqq \lim_{n\to+\infty} p_n
\)
exists. 

We move on to showing that $p^*$ satisfies the system in Equation \eqref{eq:lowhitprob2}. Fix any $n \in \N$. Then
\begin{align}
    &||p^* - (\one_A + \one_{A^c\setminus \A_\T} \cdot \underlineT p^*)|| \nonumber \\
    & \qquad \le ||p^* - p_{n+1}|| + ||p_{n+1} - (\one_A + \one_{A^c\setminus \A_\T} \cdot T_{n+1} p_{n+1})|| \nonumber \\
    &\qquad\qquad\qquad+ ||\one_{A^c\setminus \A_\T} \cdot T_{n+1} p_{n+1} - \one_{A^c\setminus \A_\T} \cdot \underlineT p^*||\nonumber \\
    & \qquad  = ||p^* - p_{n+1}|| + ||\one_{A^c\setminus \A_\T} \cdot T_{n+1} p_{n+1} - \one_{A^c\setminus \A_\T} \cdot \underlineT p^*||\nonumber\\
    & \qquad  \le ||p^* - p_{n+1}|| + ||T_{n+1} p_{n+1} - \underlineT p^*||,\label{pennablu1}
\end{align}
for any vector norm $||\cdot||$.
Continuing \eqref{pennablu1} 
\begin{align}
    &||p^* - (\one_A + \one_{A^c\setminus \A_\T} \cdot \underlineT p^*)|| \nonumber \\
    &\qquad \le ||p^* - p_{n+1}|| + ||T_{n+1} p_{n+1} - \underlineT p_n|| + ||\underlineT p_n - \underlineT p^*||\nonumber\\
    & \qquad\le ||p^* - p_{n+1}|| + ||T_{n+1} p_{n+1} - T_{n+1} p_n|| + ||p_n - p^*||\nonumber\\
    & \qquad \le ||p^* - p_{n+1}|| + ||p_{n+1} - p_n|| + ||p_n - p^*||,\label{pennablu2}
\end{align}
where we used Lemma \ref{lemma:proptransop}(3) for the second and third inequality. Since all summands on the final right-hand side of \eqref{pennablu2} vanish as $n \to +\infty$, we find that $p^*$ is a solution of the system in Equation \eqref{eq:lowhitprob2}. Since Proposition \ref{prop:423} guarantees uniqueness of the solution, it follows that $p^*=\ulp^{\T}$.
\end{proof}

\subsection{Computing Upper Hitting Probabilities}
We present an algorithm to compute the upper hitting probability \(\olp^\T\), analogous to the one for $\ulp^\T$. As with lower hitting probabilities, the algorithm must be initialised with a matrix $T_1$ satisfying $\C_{T_1}=\W_\T$. Any interior point $T_1\in \T$ satisfies this condition because $T_1(x,y)>0$ if and only if $[\overlineT \one_{\{y\}}](x)>0$.
\begin{prop}
\label{mainupprop}
Let $T_1 \in \T$ be any transition matrix such that $\C_{T_1} = {\W_{\T}}$. For all $n \in \N$, let $p_n$ be the vector of hitting probabilities for $T_n$, i.e. the unique solution of the linear system \[p_n = \one_A + \one_{A^c\setminus \C_{T_n}} \cdot T_n p_n,\] and let $T_{n+1}$ be an extreme point of $\T$ such that \[T_{n+1} p_n=\overlineT p_n\] and such that $T_{n+1}(x,\cdot)=T_n(x,\cdot)$ if 
$[T_{n}p_n](x)=[\overlineT p_n](x)$.
Then, the sequence $(p_n)_{n \in \N}$ is nondecreasing, and its limit $p^* \coloneqq \lim_{n \to +\infty} p_n$ is the vector of upper hitting probabilities, i.e the minimal nonnegative solution of the system 
\begin{equation}\label{eqimpo}
    \olp^{\T}=\one_A+\one_{A^c\setminus \W_\T}\cdot \overlineT \olp^{\T}.
\end{equation}
\end{prop}
We delay the proof of this proposition until after the following lemma.
\begin{lemma}\label{inclu_lemma}
    For the sequence of transition matrices $(T_n)_{n\in \N}$ defined in Proposition \ref{mainupprop} it holds that
    \begin{equation*}
        \C_{T_n}=\W_\T
    \end{equation*}
    for all $n\in \N$. Equivalently, $\C_{T_n}\supseteq \C_{T_{n+1}}$.
\end{lemma}
\begin{proof}
    Let $x \notin  \C_{T_n}$. We show that $x \notin \C_{T_{n+1}}$. If $x \in A$, then trivially $x \notin \C_{T_{n+1}}$, so assume $x \in A^c$. Then, since $x \notin\C_{T_n}$, we have $x \to^{T_n} y$ for some $y \in A$. By Lemma \ref{lemma:419}, there exists $m\in \N$ and a sequence of states 
\[
x = x_0, \ldots, x_m = y
\]
such that $T_n(x_{k}, x_{k+1}) > 0$ and $p_n(x_{k}) \leq p_n(x_{k+1})$ for all $k \in \{0, \ldots, m-1\}$. 

If it happens that $[T_{n+1}p_n](x_k) = [T_n p_n](x_k)$ for all $k \in \{0, \ldots, m\}$, then it clearly also holds that $T_{n+1}(x_k, x_{k+1}) > 0$ for all $k \in \{0, \ldots, m-1\}$, and therefore $x \to^{T_{n+1}} y$ and $x \notin \C_{T_{n+1}}$.

If $[T_{n+1}p_n](x_k) > [T_n p_n](x_k)$ for some $k \in \{0, \ldots, m\}$, let $k^*$  be the index of the first such state in the sequence $x_0, \ldots, x_m$. If $x_{k^*} \in A$, then, similarly to the previous case, we have $x \to^{T_{n+1}} x_{k^*} \in A$ and we are done. Hence, we assume that $x_{k^*} \in A^c$. Define
\[
\SSS \coloneqq \{z \in \X : p_n(z) > p_n(x_{k^*})\}.
\]
Let us assume \emph{ex absurdo} that $\SSS = \emptyset$. Since $p_n(z) = 1 \geq p_n(x_{k^*})$ for $z \in A$, $\SSS = \emptyset$ occurs only if $p_n(x_{k^*}) = 1$. In this case, we find
\[
1 \geq [T_{n+1}p_n](x_{k^*}) > [T_n p_n](x_{k^*}) = p_n(x_{k^*}) = 1,
\]
which is a contradiction, so $\SSS \neq \emptyset$.

Next, assume \emph{ex absurdo} that for all $z \in \SSS$, it holds that \[T_{n+1}(x_{k^*}, z) = 0.\] Then
\begin{align*}
[T_{n+1}p_n](x_{k^*}) &=\sum_{z\in \X} T_{n+1}(x_{k^*}, z)p_n(z)\\
&\leq \sum_{z\notin \SSS} T_{n+1}(x_{k^*}, z)p_n(x_{k^*})\\
&= p_n(x_{k^*})\underset{=1}{\cancel{\sum_{z\notin \SSS} T_{n+1}(x_{k^*}, z)}},
\end{align*}
where we used $p_n(z)\le p_n(x_{k^*})$ for all $z\notin \SSS$.
Since $[T_{n}p_n](x_{k^*}) = p_n(x_{k^*})$, we find $[T_{n+1}p_n](x_{k^*}) \leq [T_{n}p_n](x_{k^*})$ which is a contradiction. Therefore, there exists at least one $z \in \SSS$ such that $T_{n+1}(x_{k^*}, z) > 0$. Moreover, since $x \to^{T_{n+1}} x_{k^*}$ and $T_{n+1}(x_{k^*}, z) > 0$, we have $x \to^{T_{n+1}} z$, for some $z\in \SSS$.

So far, we have established that $x \to^{T_{n+1}} z$, with $p_n(x) < p_n(z)$. If $z \in A$, then $x \notin \C_{T_{n+1}}$ and the proof is complete. Otherwise, we find ourselves again in the starting situation, as we want to show that $z \to^{T_{n+1}} A$. The difference is that we have a strictly higher hitting probability of the state we are considering. The state $z$ is again in one of the two cases we have already described: if it is in the first case, then $z \to^{T_{n+1}} y$ for some $y \in A$ and therefore also $x \to^{T_{n+1}} y$, which implies $x \notin \C_{T_{n+1}}$. Otherwise, we find some $w \in \X$ with $w\ne x$ such that $p_n(x) < p_n(z) < p_n(w)$ and $z \to^{T_{n+1}} w$. Since we either establish $x \notin \C_{T_{n+1}}$ or we strictly increase the hitting probability, using the finiteness of the state space $\X$, we are able to conclude that $x\to^{T_{n+1}} A$ and $x \notin \C_{T_{n+1}}$. 
\end{proof}
\begin{proof}[Proof of Proposition \ref{mainupprop}]
This proof follows the structure of the proof of Proposition \ref{prop:424algolow}.

The sequence $(p_n)_{n\in\N}$ is well-defined since $p_n = \one_A + \one_{A^c\setminus \C_{T_n}} \cdot T_n p_n$ has a unique solution (Proposition \ref{prop:418}) and the existence of the extreme point $T_{n+1}$ is guaranteed by Proposition \ref{prop:225}.

By Proposition \ref{prop:418}, we have that, for all $n \in \N$, $p_n(x) = 1$ for all $x \in A$ and $p_n(x) = 0$ for all $x \in \C_{T_n}=\W_\T$, therefore $(p_n)_{n\in\N}$ is trivially convergent on $ A\cup \W_\T $.

Next, we show that $(p_n)_{n \in \N}$ is nondecreasing and convergent on $A^c \setminus {\W_{\T}}$. 
Recall that $T_{n+1} p_n = \overlineT p_n \geq T_n p_n$ for all $n\in \N$.
This yields
\begin{align*}
p_n |_{A^c \setminus {\W_{\T}}} &= (T_n p_n) |_{A^c \setminus {\W_{\T}}} \\
&\leq (T_{n+1} p_n) |_{A^c \setminus {\W_{\T}}} \\
&= T_{n+1} |_{A^c \setminus {\W_{\T}}} p_n |_{A^c \setminus {\W_{\T}}} + (T_{n+1} \one_A) |_{A^c \setminus {\W_{\T}}},
\end{align*}
where we used Lemma \ref{lemma:rewritehitprob} and the fact that  $\W_\T = \C_{T_n}$ for all $n \in \N$. Repeatedly expanding this inequality $m \in \N$ times yield
\begin{align*}
p_n |_{A^c \setminus {\W_{\T}}} &\leq (T_{n+1} |_{A^c \setminus {\W_{\T}}})^m p_n |_{A^c \setminus {\W_{\T}}} 
+ \left[\sum_{k=0}^{m-1} (T_{n+1} |_{A^c \setminus {\W_{\T}}})^k\right] (T_{n+1} \one_A) |_{A^c \setminus {\W_{\T}}}.
\end{align*}
Letting $m \to +\infty$ and applying Lemma \ref{lemma:417} yields 
\begin{align*}
p_n |_{A^c \setminus {\W_{\T}}} &\leq (I - T_{n+1})^{-1} (T_{n+1} \one_A) |_{A^c \setminus {\W_{\T}}} \\
&= p_{n+1} |_{A^c \setminus {\W_{\T}}},
\end{align*}
where the last step follows from Proposition~\ref{prop:418}. Thus, the sequence $(p_n)_{n \in \N}$ is nondecreasing.

Since $(p_n)_{n\in\N}$ is nondecreasing and $p_n \in [0,1]^{{\X}}$ for all $n \in \N$, it follows that the limit 
\(
p^* \coloneqq \lim_{n\to+\infty} p_n
\)
exists.

We move on to showing that $p^*$ satisfies the system in Equation \eqref{eqimpo}. Fix any $n \in \N$. Then
\begin{align}
&||p^* - (\one_A + \one_{A^c\setminus\W_\T}\cdot \overlineT p^*)||\nonumber \\
& \qquad\le ||p^* - p_{n+1}||  + ||p_{n+1} - (\one_A + \one_{A^c\setminus\W_\T}\cdot T_{n+1} p_{n+1})|| \nonumber\\
&\qquad\qquad\qquad + ||\one_{A^c\setminus \W_\T}\cdot T_{n+1} p_{n+1} - \one_{A^c\setminus\W_\T}\cdot \overlineT p^*|| \nonumber\\
&\qquad= ||p^* - p_{n+1}|| + ||\one_{A^c\setminus\W_\T}\cdot T_{n+1} p_{n+1} - \one_{A^c\setminus\W_\T}\cdot \overlineT p^*|| \nonumber\\
&\qquad\le ||p^* - p_{n+1}|| + ||T_{n+1} p_{n+1} - \overlineT p^*||, \label{eq:finestra1}
\end{align}
for any vector norm $||\cdot||$.
Continuing \eqref{eq:finestra1}  
\begin{align}
&||p^* - (\one_A + \one_{A^c\setminus\W_\T}\cdot \overlineT p^*) \nonumber|| \\
&\qquad\le ||p^* - p_{n+1}|| + ||T_{n+1} p_{n+1} - \overlineT p_n||+ ||\overlineT p_n - \overlineT p^*|| \nonumber \\
&\qquad\le ||p^* - p_{n+1}|| + ||T_{n+1} p_{n+1} - T_{n+1} p_n|| + ||p_n - p^*|| \nonumber \\
&\qquad\le ||p^* - p_{n+1}|| + ||p_{n+1} - p_n|| + ||p_n - p^*||, \label{wok1}
\end{align}
where we used Lemma \ref{lemma:proptransop}(3) for the second and third inequalities.
Since all summands on the final right-hand side of \eqref{wok1} vanish as $n\to +\infty$, we find that $p^*$ is a solution of the system in Equation \eqref{eqimpo}.

It remains to establish that $p^* = \olp^{\T}$, i.e. that $p^*$ is the \textit{minimal} nonnegative solution of the system in Equation \eqref{eqimpo}.  We already know that they agree on $A\cup {\W_{\T}}$, so fix an arbitrary state $x\in A^c\setminus {\W_{\T}}$. 
We also know that $\olp^{\T}(x)=\sup_{T\in\T} p^T(x)$ thus $p_n(x)\le \olp^{\T}(x)$ as $p_n$ is the hitting probability associated with $T_n\in \T$. Since we have established that $p_n(x) \to p^*(x)$ we conclude that $p^*(x)\le \olp^{\T}(x)$, and, by minimality of $\olp^\T$, $p^*=\olp^\T$.

\end{proof}
\subsubsection{Remark - the Tiebreaking Condition}
    Proposition \ref{mainupprop} is similar to Proposition \ref{prop:424algolow}, but imposes an additional condition on the next transition matrix:
    \begin{equation}\label{eq:condupperprob}
         T_{n+1}(x,\cdot)=T_n(x,\cdot) \ \text{ if } \ [T_{n}p_n](x)=[\overlineT p_n](x).
    \end{equation}
    In other words, whenever the current row $T_n(x,\cdot)$ already attains the maximum $[T_{n} p_n](x)=\max_{T\in \T} [T p_n](x)$, the algorithm leaves that row unchanged. Equivalently, the condition acts as a tiebreaker: if several distributions achieve the maximum and $T_n(x,\cdot)$ is one of them, we keep $T_n(x,\cdot)$.
    This prevents pathological behaviours as illustrated in the following example. 
\begin{ex}
    Let $\X=\{1,2,3\}$ be the state space, and let $\{3\}$ be the target set. Let the set of transition matrices on $\X$ be
    \begin{equation*}
        \T=\ch\left\{T=
        \begin{bmatrix}
            1 & 0 & 0\\
            q_1 & q_2 & q_3\\
            0 & 0 & 1\\
        \end{bmatrix}
        : q_i\in [0,1],q_1+q_2+q_3=1
        \right\}.
    \end{equation*}
    Initialize the algorithm with
    \[
    T_1=\begin{bmatrix}
            1 & 0 & 0\\
            0 & 0 & 1\\
            0 & 0 & 1\\
        \end{bmatrix}
    \]
    which has $\C_{T_1}=\W_\T=\{1\}$. This matrix maximises the hitting probability having $\olp^\T =p_1=\left(0,1,1\right)^\top$. Any transition matrix imposing $q_1=0$ achieves the maximum in $\max_{T\in \T} [T p_1](2)=q_2+q_3=1$, so one possible choice for the extreme point $T_2$ is 
    \[
    T_2=\begin{bmatrix}
            1 & 0 & 0\\
            0 & 1 & 0\\
            0 & 0 & 1\\
        \end{bmatrix}.
    \]
    However, $T_2$ disconnects state $2$ from the target set and its associated vector of hitting probabilities is $p_2=\left(0,0,1\right)^\top$. Therefore, without the additional tiebreaking condition in Equation \eqref{eq:condupperprob} the algorithm might fail to converge.  
\end{ex}
\subsection{Interpretation as an MDP and Connection with Policy Iteration}
The proposed algorithms can be seen as a variation of the policy iteration algorithm for Markov Decision Processes (MDPs)~\cite{Puterman1994, SuttonBarto2018}. In particular, for every state $x\in \X$ the set $\T_x$ can be seen as a set of actions. Then, every row of a transition matrix $T$, $T(x,\cdot)$, is an admissible action in state $x$.
Moreover, every transition matrix $T\in \T$ can be seen as a policy $\pi_T: \X \to \T$. Taking the reward function $r$ as the hitting indicator $\one_A$, we find that, for every policy $\pi_T$, the value function is the hitting probability $p^T$. The optimal value $v^*=\ulp^\T=\inf_{T\in \T} p^T$ satisfies the (undiscounted) Bellman optimality equation~\cite{SuttonBarto2018} (which is equivalent to Equation \eqref{eq:lowhitprob}). The optimal value and the optimal policy can be computed using a policy iteration algorithm~\cite{Howard1960} which, similar to the algorithm in Proposition \ref{prop:424algolow}, alternates between policy evaluation, i.e. solving the linear system, and policy improvement. 

These connections between imprecise Markov chains and MDPs have also previously been described in the literature; see, for instance, comments by Krak et al.~\cite{krak2019hitting, krak2021comphit} about these relations, as well as for interpretations of similar algorithms for IMCs as being versions of value- and policy iteration methods for MDPs.

\section{Practical Considerations and Complexity Analysis}\label{sec:practical}
It remains to discuss how the algorithms described in Proposition \ref{prop:424algolow} and Proposition \ref{mainupprop} can be used in practice. This consists of three main aspects, namely how many iterations are required to reach convergence, the complexity of any such iteration, and the computational cost of finding suitable starting points $T_{1} \in \mathcal{T}$. We assume that the set $\mathcal{T}$ is explicitly given, while the operators $\underline{T}$ and $\overline{T}$ are implicitly defined through Definition \ref{defn:lowoper42}. This is the situation most common in practice, as the set $\mathcal{T}$ often characterises the uncertainty surrounding the true process by providing error bounds on the true transition probabilities. 
 
\subsection{Time per Iteration}
The computation performed for each iteration in Proposition \ref{prop:424algolow} and Proposition \ref{mainupprop} can be divided into two parts. Firstly, we need to find the next element in the sequences $\left(p_{n}\right)_{n \in \mathbb{N}}$. By Proposition \ref{prop:418}, this requires us to solve a linear system of equations, which can be done in at most $\mathcal{O}\left(N^{c}\right)$ time~\cite{Demmel2007}: the time required to multiply two $N \times N$ matrices. Currently, the best (published) estimate of $c$ is $c=2.371339$ \cite{alman2025more}. 

Secondly, we investigate the complexity of computing the next element of the sequences $\left(T_{n}\right)_{n \in \mathbb{N}}$.
Finding such an extreme point requires solving the optimisation problem from Definition \ref{defn:lowoper42} and, therefore, strongly depends on the definition of $\mathcal{T}$.
Typically, this computation takes at least $\mathcal{O}(N^{3})$ time~\cite{krak2021comphit}. In the remainder of the analysis, we denote the complexity of finding an extreme point $T \in \mathcal{T}$ such that $\underline{T} f=T f$ or $\overline{T} f=T f$ for some arbitrary $f \in \mathbb{R}^{\mathcal{X}}$ by $\mathcal{O}(\operatorname{Extr}(\mathcal{T}, \mathcal{X}))$. Then, for Proposition \ref{prop:424algolow} and Proposition \ref{mainupprop}, the time per iteration is at most $\mathcal{O}\left(N^{c}+\operatorname{Extr}(\mathcal{T}, \mathcal{X})\right)$.

\subsection{Starting Points}
We now discuss the complexity of finding a suitable $T_{1} \in \mathcal{T}$ for Proposition \ref{prop:424algolow} and Proposition \ref{mainupprop}.  
\subsubsection{Lower Hitting Probabilities}
As already presented in Section \ref{subsec:computing_lower}, definition (LR1) provides a practical algorithm to identify the set $\A_\T$. 
The computational cost of this algorithm is $\mathcal{O}\left(\operatorname{Extr}(\mathcal{T}, \mathcal{X})N\right)$. This follows from the fact that for each $z\in \X\setminus \D_k$ we need to check whether $[\underlineT\one_{\D_k}](z)>0$: this operation is executed $\mathcal{O}(N)$ times.
 
Having obtained $\A_{\T}$, it remains to identify a matrix $T_{1} \in \mathcal{T}$ such that $\mathcal{C}_{T_{1}}=\A_{\T}$. In the last iteration of the algorithm, we find $[\underlineT \one_{\D_{k^*}}](z)=0$ for all $z\in \A_\T$. Hence, we also have $|\A_\T|$ transition matrices, one for each $z\in \A_\T$, satisfying $[S_z\one_{\D_{k^*}}](z)=[\underlineT \one_{\D_{k^*}}](z)=0$. We then construct $T_1$ row-wise by setting $T_1(z,\cdot)=S_z(z,\cdot)$.
In summary, identifying $\A_{\T}$ and finding $T_{1} \in \mathcal{T}$ to use in Proposition \ref{prop:424algolow} requires at most $\mathcal{O}\left(\operatorname{Extr}(\mathcal{T}, \mathcal{X})N\right)$ time.

\subsubsection{Upper Hitting Probabilities}
The matrix $T_1\in \T$ with $\C_{T_1}=\W_\T$ can be constructed row-wise.
In particular, for every $x\in \X$, choose any $T_1(x,\cdot)$ in the relative interior of $\T_x$, taken with respect to the affine hull of $\T_x$.
In this way, $T_1(x,y)>0$ for all $y\in \X$ such that $[\overlineT\one_{\{y\}}](x)>0$.
The computational cost of this operation is $\mathcal{O}(N)$.
 
\subsection{Number of Iterations}
We now investigate the number of iterations needed for convergence. We start by defining a sequence of vectors $\left(p_{n}\right)_{n \in \mathbb{N}}$ to be {strictly decreasing} if $p_{n+1} \leq p_{n}$ and $p_{n+1} \neq p_{n}$ for all $n \in \mathbb{N}$ and {strictly increasing }if $p_{n+1} \geq p_{n}$ and $p_{n+1} \neq p_{n}$. We then have the following result.
\begin{cor}\label{cor:431}
Let $(p_{n})_{n\in\mathbb{N}}\subset[0,1]^{\X}$ be the sequence constructed in Proposition~\ref{prop:424algolow}. Then $(p_n)$ is either strictly decreasing everywhere, or there exists $n^*\in\mathbb{N}$ such that $(p_n)$ is strictly decreasing for all $n< n^*$ and $p_k=p^*$ for all $k\ge n^*$, where $p^*=\lim_{n\to\infty}p_n$.

Similarly, let $(p_{n})_{n\in\mathbb{N}}\subset[0,1]^{\X}$ be the sequence constructed in Proposition~\ref{mainupprop}. Then $(p_n)$ is either strictly increasing everywhere, or there exists $n^*\in\mathbb{N}$ such that $(p_n)$ is strictly increasing for all $n< n^*$ and $p_k=p^*$ for all $k\ge n^*$, where $p^*=\lim_{n\to\infty}p_n$.
\end{cor}
\begin{proof}
    Completely analogous to the proof of Krak's Corollary 6 \cite{krak2021comphit}.
\end{proof} 

If we now assume that $\mathcal{T}$ has at most $m \in \mathbb{N}$ extreme points, we can use Corollary \ref{cor:431} to find an upper bound on the number of iterations performed in Proposition \ref{prop:424algolow} and Proposition \ref{mainupprop}.

\begin{cor}\label{cor:432}
Let $\mathcal{T}$ have $m\in\mathbb{N}$ extreme points. Then the sequence $(p_n)_{n\in\mathbb{N}}$ constructed in Proposition~\ref{prop:424algolow} stabilises in at most $m$ steps: there exist $n^*\in\mathbb{N}$ with $n^*\le m$ such that $p_k=p^*$ for all $k\ge n^*$. 

The same conclusion holds for the sequence constructed in Proposition~\ref{mainupprop}.
\end{cor}
\begin{proof}
    Completely analogous to the proof of Krak's Corollary 7 \cite{krak2021comphit}. 
\end{proof}
The bound established by Corollary~\ref{cor:432} is tight, i.e. there are instances in which the presented algorithms need to cycle over every extreme point of $\T$ before reaching the optimal transition matrix. 
\begin{ex}\label{ex:worstcase}
Let $\X=\{1,2,3\}$ be the state space, and let $\{2\}$ be the target set. Let the set of transition matrices be
\begin{equation*}
    \T(m)=\ch\left\{T_n=
    \begin{bmatrix}
        1-\frac{1}{n}-\frac{1}{n^2} & \frac{1}{n} & \frac{1}{n^2}\\
        0 & 1 & 0\\
        0 & 0 & 1\\
    \end{bmatrix}
    : n=2^k, k\in \N, 1\le k\le m
    \right\},
\end{equation*} 
where $m\in \N$ is the number of extreme points of $\T(m)$. The situation is represented by the following graph:

\begin{minipage}{0.48\textwidth}
    \begin{center}
    \begin{tikzpicture}[
        xscale=1.75,yscale=1.75,
        every node/.style={draw=black,circle,fill=black!70,inner sep=2pt},
        redNode/.style={draw=red, fill=red!50},
        blueNode/.style={draw=blue, fill=blue!55},
        blackNode/.style={draw=black, fill=black},
        every label/.style={rectangle,fill=none,draw=none},
        every edge/.style={draw,bend right,looseness=0.3,
            postaction={
                decorate,
                decoration={
                    markings,
                    mark=at position 0.8 with {\arrow[#1]{Stealth}}
                }
            }
        }
    ]
        \node[blueNode,label=180:\textbf{1}] (1) at (2,54.4) {};
        \node[blueNode,label=180:\textbf{2}] (2) at (1.15,55.6) {};
       \node[blueNode,label=180:\textbf{3}] (3) at (2.85,55.6) {};
        \path (1) edge[bend left=0] node[left, draw=none, fill=none, xshift=0.9pt] {$\nicefrac{1}{n}$} (2);
        \path (1) edge[bend right=0] node[right,draw=none, fill=none, yshift=0.1pt] {$\nicefrac{1}{n^2}$} (3);
        \path (1) edge[relative=false,out=270,in=345,min distance=3.5ex] node[right,draw=none, fill=none, yshift=0.1pt] {$1-\nicefrac{1}{n}-\nicefrac{1}{n^2}$} (1);
        \path (2) edge[relative=false,out=15,in=90,min distance=3.5ex] node[right,draw=none, fill=none, yshift=-0.1pt] {$1$} (2);
        \path (3) edge[relative=false,out=15,in=90,min distance=3.5ex] node[right,draw=none, fill=none, yshift=-0.1pt] {$1$} (3);
    \end{tikzpicture}
\end{center}
\end{minipage}
\begin{minipage}{0.48\textwidth}
        \centering
        \includegraphics[width=0.69\linewidth]{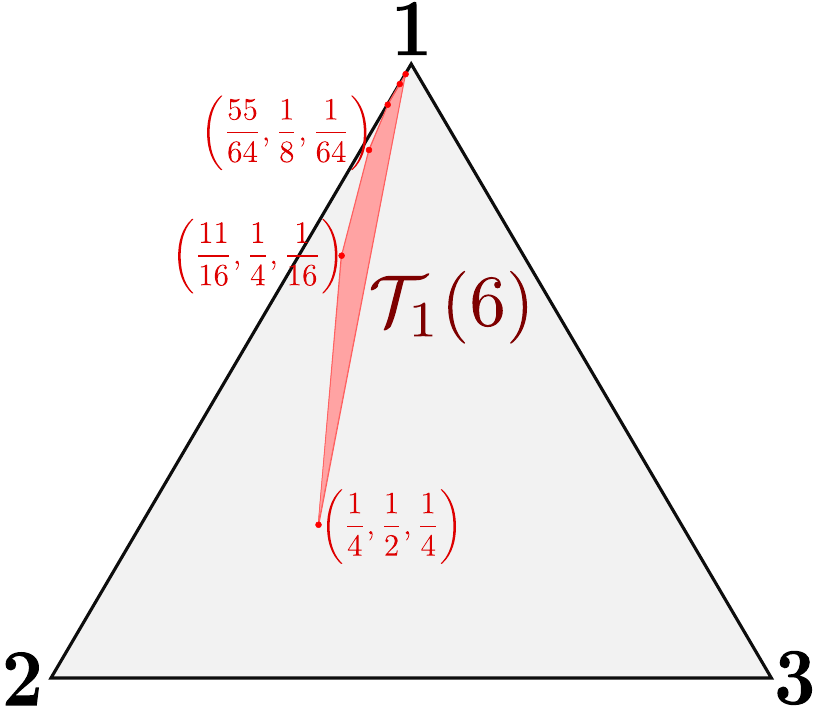}
        \captionof{figure}{Credal set of $\T_1(6)$.}
        \label{fig:credalset_worstcase}
\end{minipage}
    
Given $k\in \N$ and $n=n(k)=2^k$, we find that the vector of hitting probabilities associated with $T_{n(k)}$ equals $p_{n(k)}=(\nicefrac{n}{n+1}, 1, 0)$. We claim that the matrix $T_{n(k+1)}$ satisfies $T_{n(k+1)}= \argmax_{T\in \T} T  p_{n(k)}$ for all $1\le k\le m-1$, i.e. 
\begin{align*}
    k+1=\argmax_{j\in \N, 1\le j\le m} T_{n(j)} p_{n(k)}.
\end{align*}
Expanding this expression:
\begin{align*}      
    \argmax_{j\in \N, 1\le j\le m} T_{n(j)} p_{n(k)}&=\argmax_{j\in \N, 1\le j\le m} \left(1-2^{-j}-2^{-2j} , 2^{-j}, 2^{-2j}\right) \cdot 
    \begin{pmatrix}
        \frac{2^k}{2^k +1}\\
        1\\
        0
    \end{pmatrix}\\
    &=\argmax_{j\in \N, 1\le j\le m} \left(\frac{2^k}{2^k +1}(1-2^{-j}-2^{-2j})+2^{-j} \right)\\
    &=\argmax_{j\in \N, 1\le j\le m} \left(2^{-j}\left(1-\frac{2^k}{2^k +1}\right) - 2^{-2j}\frac{2^k}{2^k +1} \right)\\
    &=\argmax_{j\in \N, 1\le j\le m}\left(2^{-j}\left(\frac{1}{2^k +1}-2^{-j}\frac{2^k}{2^k +1} \right)\right).
\end{align*}
The maximum is achieved when $2^{-j}=\frac{2^{-k}}{2}=2^{-(k+1)}$, so when $j=k+1$, which is what we wanted to show.

Initialize the algorithm for upper hitting probabilities with the matrix $S_1\coloneqq T_{n(1)}$. After one iteration, we get $S_2=T_{n(2)}$, as \[S_2 p_{n(1)}=\max_{T\in \T} Tp_{n(1)}=\overlineT p_{n(1)},\] and so on. Thus, only after cycling over all $m$ extreme points of $\T(m)$ we eventually find the maximising matrix in $S_m=T_{n(m)}$. 
\end{ex}
An analogous example can be given for lower hitting probabilities.
Despite these examples, the average behaviour differs from the worst-case scenario that we just presented; the following section showcases this.
\section{Numerical Analysis}\label{sec:numanal}
In this section, we present the results of numerical experiments designed to showcase the average behaviour of the presented algorithms, namely the fact that the average number of iterations is significantly lower than the worst-case scenario illustrated in Example \ref{ex:worstcase}. 

\subsection{Description of the Experimental Setup}
We explain the experimental setup. Given $N$ the number of states of $\X$, we set $A=\{1\}$ to be the target state.  First, we construct a random graph with $N$ nodes, where each directed edge indicates a positive upper probability of transition between the corresponding states. Then, we use this graph to generate the set of transition matrices $\T$.
We construct the random graph in the following three steps.
\begin{itemize}
    \item State $N$ is made absorbing, i.e. any random walk that reaches state $N$ remains there with probability one. We impose this condition so that the upper hitting probability can be less than one. 
    \item We construct a random tree rooted in $A=\{1\}$. Specifically, let $\pi$ be a random permutation of $\X\setminus (A\cup \{N\})$; for each state $x$ in the order given by $\pi$, we draw a directed edge from $x$ to a state chosen uniformly at random from the set
    $A\cup \{y\mid \pi(y) \le \pi(x)\}$.
    This construction is chosen so that every state but $N$ has an upper hitting probability greater than zero.
    \item For each pair of nodes, $x\in \X\setminus (A\cup \{N\})$ and $y\in \X$, we add a directed edge from $x$ to $y$ with probability $\nicefrac{(\lambda-1)}{(N-1)}$, where $1\le \lambda \le N$. This Erd\H{o}s-Rényi-like~\cite{erdos1960evolution} construction ensures that each node has an average  degree of $\lambda$. In this way, we can easily control the graph's density and study its influence on the average number of iterations.
\end{itemize}
We now generate the set of transition matrices $\T$. 
To do that, we consider the following two methods:
\begin{enumerate}
    \item Given a parameter $0<\varepsilon<1$, we sample a transition matrix $\tilde{T}$ uniformly at random\footnote{In particular, for each $x\in \X\setminus (A\cup \{N\})$, we sample a point $p_x$ uniformly at random from the $n_x -1$--dimensional simplex, and then we construct $\tilde{T}$ row-wise, i.e. $\tilde{T}|_{E_x}(x,\cdot) = p_x$, where $E_x$ is the set of states $x$ is connected to.} and we $\varepsilon$-contaminate it~\cite{Huber1964,itip:specialcases}. In particular, the set of transition matrices is
    \[
    \T=\{(1-\varepsilon)\tilde{T}+\varepsilon S \mid \text{ for all trans. mat $S$}\}.
    \]
    For numerical experiments, we set $\varepsilon=0.1$; experiments with other values exhibited a similar behaviour.
    \item For each $x\in \X\setminus (A\cup \{N\})$, we sample $n_x$ points\footnote{We choose this number of points to be sampled so that $\T_x$ has the same dimension as the simplex in which it is embedded.} uniformly at random from the $n_x-1$-dimensional simplex, where $n_x$ is the number of outgoing edges of $x\in \X$, and $\T_x$ is their convex hull. Since the set of transition matrices $\T$ has separately specified rows, $\T$ is the Cartesian product of all $\T_x$.
\end{enumerate}
\subsection{Numerical Results}
Since the algorithms from Propositions \ref{prop:424algolow} and \ref{mainupprop} exhibit nearly identical iteration counts, we will often present the number of iterations without explicitly specifying which algorithm it refers to. 
Figure \ref{fig:imagetablesparsity} displays the average number of iterations, on the vertical axis, as a function of the average degree $\lambda$ of the graph, on the horizontal axis. The graphs contain either 10 (left) or 50 (right) nodes, and the set of transition matrices is generated using $\varepsilon$-contamination (top) or by taking the convex hull of random samples (bottom). The averages are obtained by running 1000 independent simulations.

\begin{figure}[h]
    \centering
    \includegraphics[width=\linewidth]{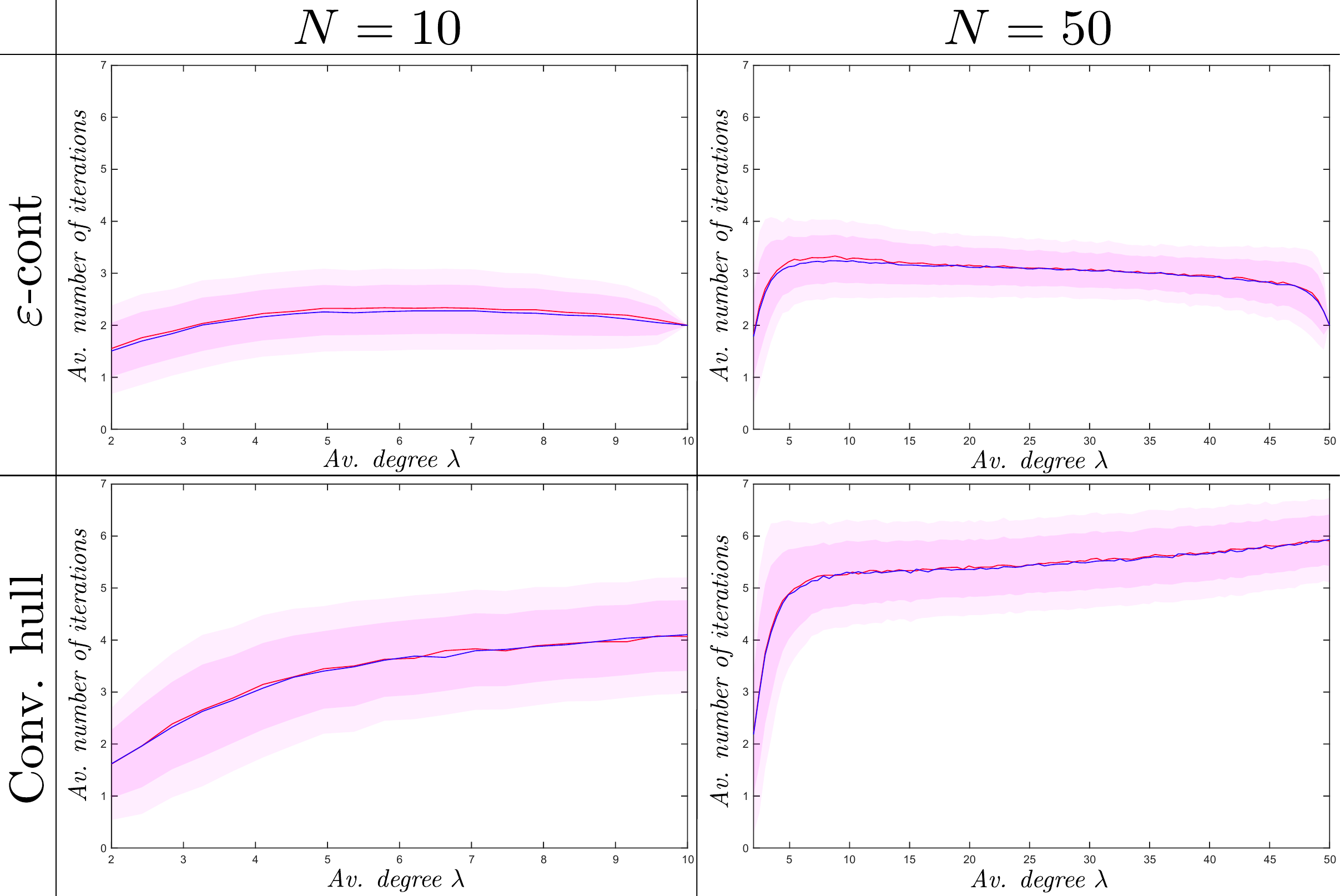}
    \caption{The red and blue lines represent the average number of iterations needed to compute lower and upper hitting probabilities, respectively, as a function of the average degree $\lambda$ of each node of the graph. The two magenta shaded regions correspond to the intervals within one standard deviation ($\pm \sigma$) and approximately $95\%$ confidence ($\pm 1.96\sigma$) around the mean number of iterations.}
    \label{fig:imagetablesparsity}
\end{figure}
Both algorithms are highly efficient, requiring only few iterations to converge. The average number of iterations never exceeds six, while the maximum observed was nine.
The histogram in Figure \ref{fig:histogram} illustrates in detail the empirical distribution of iteration counts over the 1000 simulated experiments.
\begin{figure}[h]
    \centering
    \includegraphics[width=\linewidth]{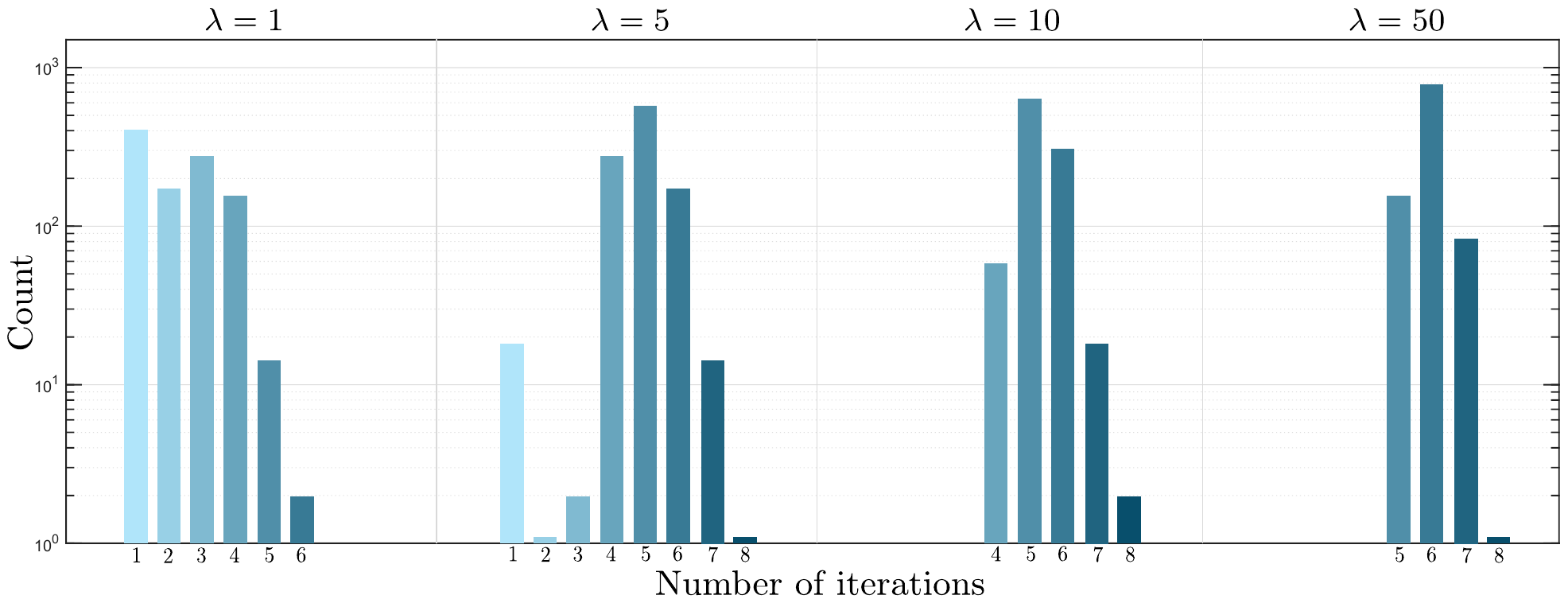}
    \caption{Number of instances, out of 1000, in which the algorithm for upper hitting probabilities converged in a given number of iterations, for different values of the average degree $\lambda$. The graph has $N=50$ nodes and the set of transition matrices is generated by taking the convex hull of transition matrices sampled uniformly at random.}
    \label{fig:histogram}
\end{figure}

Furthermore, the number of iterations for the complete graph with $N$ nodes and set of transition matrices generated via $\varepsilon$-contamination is always two. This is a consequence of the fact that the credal sets of this uncertainty model are probability intervals of width $\varepsilon$, namely
\[
    \T=\{T \text{ trans. mat} \mid (1-\varepsilon)\tilde{T}(x,y)\le T(x,y)\le (1-\varepsilon)\tilde{T}(x,y)+\varepsilon, \ \forall x,y\in \X\},
\]
for a $\tilde{T}$ sampled uniformly at random. For a node $x$ connected to the target set $A=\{1\}$, the first iteration of the algorithm that computes upper hitting probabilities assigns maximal probability of reaching $A$ in one step, i.e. 
\[
T_2(x,1)=(1-\varepsilon)\tilde{T}(x,1)+\varepsilon.
\]
The rest of the probability mass is 
\begin{align*}
1-\left[(1-\varepsilon)\tilde{T}(x,1)+\varepsilon\right] & = 1-\left[(1-\varepsilon)\left(1-\sum_{y\ne 1} \tilde{T}(x,y)\right)+\varepsilon\right] \\
&  = (1-\varepsilon)\sum_{y\ne 1} \tilde{T}(x,y) = \sum_{y\ne 1} (1-\varepsilon)\tilde{T}(x,y).
\end{align*}
Therefore, we must have
\[
T_2(x,y)=(1-\varepsilon)\tilde{T}(x,y),
\]
for all $y\in A^c$. The second (and last) iteration computes the upper hitting probability by solving the linear system and checks that no further improvement can be made. The same happens with lower hitting probabilities  with $\{N\}$ in place of $\{1\}$.

Moreover, we observe a nonmonotone relationship between the density of the graph and the average number of iterations in the $\varepsilon$-contamination case.
For fixed $\varepsilon$ and $N$, there is an intermediate average degree $\lambda^*$ that maximises the average number of iterations.
This behaviour is also shown in Figure \ref{fig:3dplot}.
\begin{figure}[h]
    \centering
    \includegraphics[width=\textwidth]{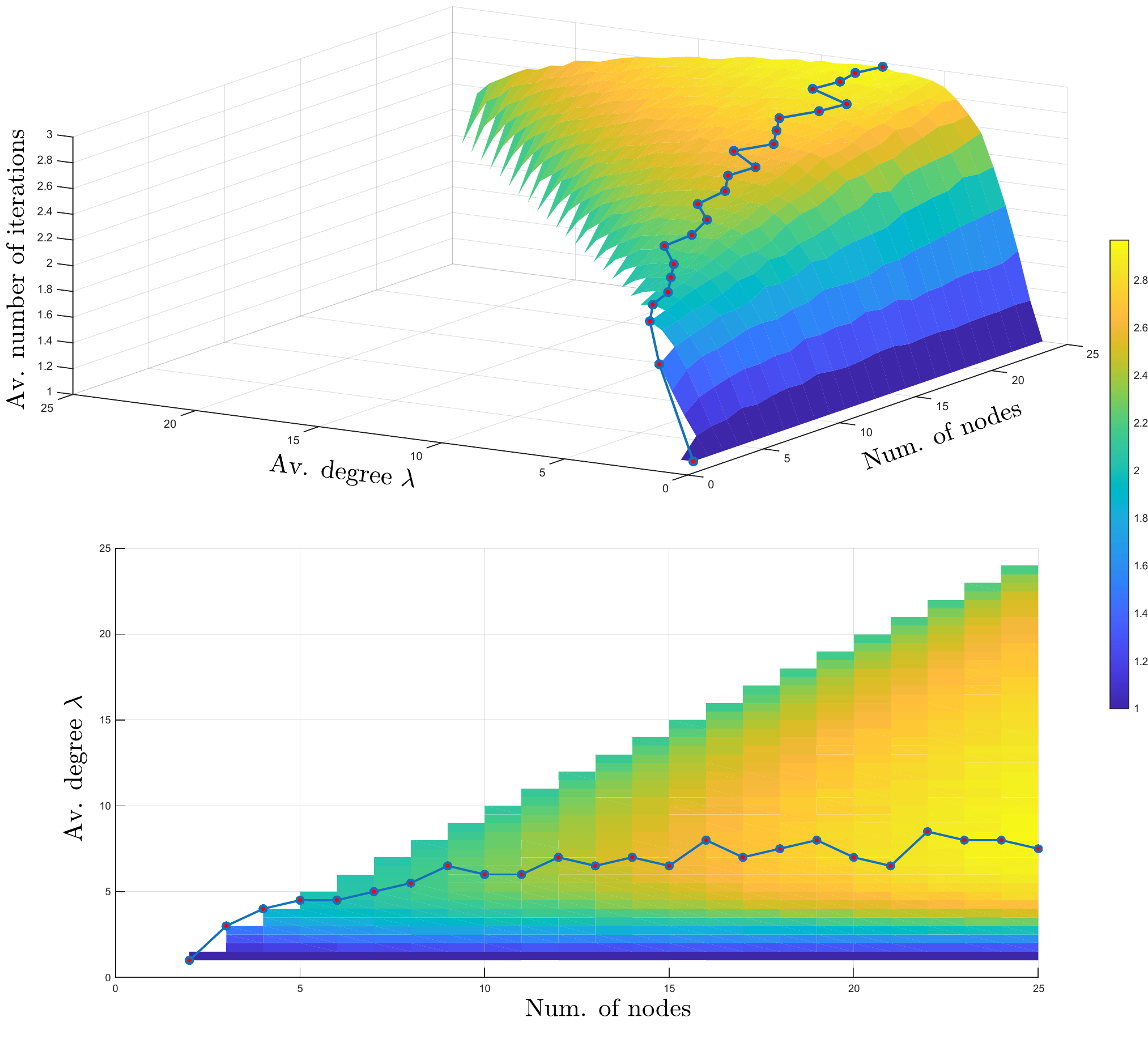}
    \caption{Average number of iterations as a function of the number of states $N$ and the average degree $\lambda$. Each red dot on the blue line marks, for a fixed number of states $N$, the average degree $\lambda^*=\lambda^*(N)$ that attains the maximum average number of iterations.}
    \label{fig:3dplot}
\end{figure}
A similar behaviour is observed when probability intervals are generated uniformly at random, resulting in arbitrary widths rather than the fixed width $\varepsilon$.
Although a formal theoretical justification is lacking, one plausible explanation for this empirical behaviour is the ``propagation of information''.
We already noted that if the graph is fully connected, the algorithm converges in two iterations: each node $x$ ``knows'' which of its neighbours is the one to assign maximal probability, i.e. the target or the absorbing state $N$. This does not happen in sparser graphs: not all nodes are connected, so node $x$ cannot a priori determine which neighbour should be assigned the highest probability.
The following example illustrates this phenomenon.
\begin{ex}
    Let $\X=\{1,\dots,N+2\}$ be the state space, and let $A=\{1\}$ be the target set. Let the set of transition matrices be 
    \begin{equation*}
          \T=\left\{T=
          \begin{bmatrix}
              1&0 &0&\dots&\dots&\dots &0\\
              q_2&0&0&\dots&\dots&\dots & 1-q_2\\
              0&q_3&0 &\dots& \dots & 1-q_3& 0\\
              0& 0 &q_4 & 0& \dots & 1-q_4& 0\\
              \vdots&\vdots&\vdots&\vdots&\vdots&\vdots&\vdots\\
              0&\dots&\dots&0&q_{N}&1-q_{N}&0\\
              \nicefrac{1}{2}&0&\dots&\dots&\dots&0&\nicefrac{1}{2}\\
              0&\dots&\dots&\dots&\dots&0&1
          \end{bmatrix}: 
          \begin{aligned}
      &\ q_x \in [0,b],\\
      &{\nicefrac{1}{2}} < b < 1
    \end{aligned}
          \right\}.
    \end{equation*}
     The situation is represented by the following graph:
     \begin{center}
    \begin{tikzpicture}[
        xscale=1.6,yscale=1.6,
        every node/.style={draw=black,circle,fill=black!70,inner sep=2pt},
        redNode/.style={draw=red, fill=red!50},
        blueNode/.style={draw=blue, fill=blue!55},
        blackNode/.style={draw=black, fill=black},
        whiteNode/.style={draw=white, fill=white},
        every label/.style={rectangle,fill=none,draw=none},
        every edge/.style={draw,bend right,looseness=0.3,
            postaction={
                decorate,
                decoration={
                    markings,
                    mark=at position 0.7 with {\arrow[#1]{Stealth}}
                }
            }
        }
    ]
        \node[blueNode,label={[xshift=-3.4mm]45:\textbf{N+2}}] (N1) at (2,56.5) {};
        \node[blueNode,label=90:\textbf{N+1}] (N) at (3.5,56.5) {};
         \node[blueNode,label=315:\textbf{N}] (Nm1) at (8.6,55){};
         \node[whiteNode] (W) at (7.1,55){};
        \node[blueNode,label=270:\textbf{1}] (1) at (2,55) {};
        \node[blueNode,label=270:\textbf{2}] (2) at (3.5,55) {};
       \node[blueNode,label=270:\textbf{3}] (3) at (5,55) {};
       \node[blueNode,label=270:\textbf{4}] (4) at (6.5,55) {};
       
        \path (2) edge[bend left=0] node[below,draw=none, fill=none] {$q_2$} (1);
        \path (3) edge[bend left=0] node[below,draw=none, fill=none] {$q_3$} (2);
         \path (4) edge[bend left=0] node[below,draw=none, fill=none] {$q_4$} (3);
        \path (N) edge[bend left=0] node[below,draw=none, fill=none, yshift=1.4mm] {$\nicefrac{1}{2}$} (N1);
        \path (N) edge[bend left=0] node[above,draw=none, fill=none, yshift=-1mm,xshift=7mm] {$\nicefrac{1}{2}$} (1);
         \path (2) edge[bend left=0] node[below,draw=none, fill=none, xshift=7.6mm, yshift=-0.25mm] {$1-q_2$} (N1);
         \path (3) edge[bend left=0] node[right,draw=none, fill=none, yshift=-5.23mm] {$1-q_3$} (N);
         \path (4) edge[bend left=0] node[right,draw=none, fill=none, xshift=3.2mm, yshift=-1.4mm] {$1-q_4$} (N);
         \path (Nm1) edge[bend left=0] node[right,draw=none, fill=none, yshift=1mm] {$1-q_{N}$} (N);
         \path (Nm1) edge[bend left=0] node[below,draw=none, fill=none, yshift=0.3mm] {$q_{N}$} (W);
         \path (1) edge[relative=false,out=225,in=135,min distance=3.5ex] (1);
         \path (N1) edge[relative=false,out=225,in=135,min distance=3.5ex] (N1);
         \node[draw=none, fill=none] at (6.86,55) {$\ldots$};
    \end{tikzpicture}
    \end{center}
     Initialise the algorithm to compute upper hitting probabilities with any matrix $T_1$ satisfying $0<q_2<\nicefrac{1}{2}$ and $q_3=b$. Then,  $p_1(1)=1$,  $p_1(N+2)=0$, $p_1(N+1)=\nicefrac{1}{2}$, $p_1(2)<\nicefrac{1}{2}$, $p_1(3)=b p_1(2)+ (1-b)p_1(N+1)<\nicefrac{1}{2}$ and $p_1(x)\le \nicefrac{1}{2}$ for every other $x\in \X$.
    
    In the first iteration, node $2$ maximises its probability of hitting the target by setting $q_2=b$.
    Node $3$ sets $q_3=0$ because it is connected to nodes $2$ and $N+1$, and $b\cdot p_1(2)+ (1-b) p_1(N+1)<p_1(N+1)$ since $p_1(2)<p_1(N+1)=\nicefrac{1}{2}$. 
    The same argument applies for every other node $x\in \{4,\dots, N\}$: $q_x$ is set to zero because $b \cdot p_1(x-1) + (1-b)p_1(N+1)<p_1(N+1)$ which follows from $p_1(x-1)<p_1(N+1)=\nicefrac{1}{2}$.
    Henceforth, solving the linear system yields $p_2(2)=b$ and $p_2(x)=\nicefrac{1}{2}$, for all $x\in \{3,\dots, N\}$. 
    
    In the second iteration, we have $b\cdot p_2(2)+(1-b)p_2(N+1)>p_2(N+1)$ because $b=p_2(2)>p_2(N+1)=\nicefrac{1}{2}$, so node $3$ now sets $q_3=b$, yielding a hitting probability of $p_3(3)=b^2 + \nicefrac{1}{2}(1-b)$. Other values of $q_x$ remain unchanged in this iteration.

    In the $k$-th iteration, we have $b\cdot p_{k}(k)+(1-b)p_{k}(N+1)>p_{k}(N+1)=\nicefrac{1}{2}$ and $q_{k+1}$ is set to $b$. It can be shown (e.g. using induction) that $p_k(k)=\nicefrac{1}{2} +b^{k-2}(b-\nicefrac{1}{2})$ which is greater than $\nicefrac{1}{2}$ if $b>\nicefrac{1}{2}$.
    
    Therefore, the information that larger $q_x$ yields greater hitting probabilities propagates backward one node per iteration, so that only after $N$ iterations the algorithm has reached convergence.
\end{ex}

\section{Conclusions and Future Work}\label{sec:conclusion}
In this work, we have investigated lower and upper hitting probabilities for imprecise Markov chains and developed efficient iterative algorithms for their computation. 
Here we summarise the main contributions and their implications.

In Section \ref{subsec:reachability} we have discussed in detail the notion of reachability in the imprecise setting. 
The literature review highlighted several different definitions of lower reachability; hence we have explored the connections among them, presenting examples to clarify their logical relations. 
In particular, we have shown the equivalence between definitions $\operatorname{(LR0)}$ and $\operatorname{(LR1)}$; this equivalence has proved crucial, as the latter definition provides a practical algorithm to identify the states that cannot lower reach the target and thus have zero lower hitting probability.
Using $\operatorname{(LR0)}$, we have been able to partition the state space and characterise lower hitting probabilities as the unique solution of a nonlinear fixed-point equation. The analogous characterisation for upper hitting probabilities fails; we needed to add stronger assumptions on the set $\T$ to restore uniqueness of the solution $\olp^\T$. 

In Section \ref{sec:computing} we have presented novel efficient iterative algorithms that compute lower and upper hitting probabilities. These algorithms alternate between solving a linear system over nontrivial states and selecting extreme points from the set of transition matrices. Although in principle these algorithms may visit all extreme points of $\T$, this has never occurred in practice during numerical experiments. 
In fact, our experiments, described in Section~\ref{sec:numanal}, have shown the efficiency of the algorithms: for every graph structure and uncertainty model analysed, very few iterations were required for convergence.

We highlight several directions for future work that follow directly from our study:
\begin{itemize}
  \item \textit{Formal justification of the nonmonotone relation between density and number of iterations.} We would like to investigate why the algorithms for computing lower and upper hitting probabilities show a nonmonotone relation between the density $\lambda$ of the graph and the average number of iterations. For specific uncertainty models (e.g. $\varepsilon$-contamination, probability intervals, etc.), we wish to find the average degree $\lambda^*=\lambda^*(N)$ that maximises the number of iterations.
  \item \textit{Handling pathological credal sets.} Although extremely rare, some credal set shapes force the algorithms to cycle over all extreme points, as in Example \ref{ex:worstcase}. Future work should identify such cases in advance and aim to derive specialised algorithms to compute lower and upper hitting probabilities that perform better in such regimes.  
  \item \textit{Extension to continuous-time IMCs.} One relatively straightforward step forward would be to generalise the fixed-point characterisation and the computational framework discussed here to continuous-time imprecise Markov chains. 
  \item \textit{Algorithms for conservative inference.} It follows from Corollary~\ref{cor:431} that the proposed algorithms provide, at each iteration $n\in\N$ before convergence, an interval which is an inner approximation of the actual lower/upper hitting probabilities $[\ulp^\T,\olp^\T]$, and which monotonically approaches these values. An interesting avenue of research consists of finding algorithms that instead compute a conservative approximation of lower and upper hitting probabilities, that is, an outer approximation that shrinks to these target values.
  \item \textit{Lower and upper meeting probabilities.} It would be interesting to characterise and compute efficiently lower and upper \emph{meeting} probabilities, i.e. the probability that two (or more) processes are ever in the same state at the same time. This present paper, combined with Sangalli {et al.}'s recent work on meeting times for imprecise Markov chains~\cite{sangalli2025meeting}, provides a foundation for this line of research.  
  \item \textit{Study the continuity of the hitting probability map.} The discontinuity of the map $T\mapsto p^T$ prevented us from adapting the argument for lower hitting probabilities to the upper case. Studying the continuity of this map could allow us to relax the strong conditions that guarantee the uniqueness of $\olp^\T$ in Proposition \ref{prop:427}. 
    \item \textit{Relation between lower and upper hitting times and probabilities.} Finally, we aim to investigate the relation between hitting times and probabilities in the imprecise framework, and seek to establish whether knowledge of one quantity allows us to infer bounds on the other. 
    
  \end{itemize}

\section*{Acknowledgements}
This work has been partly supported by the PersOn project
(P21-03), which has received funding from Nederlandse Organisatie voor Wetenschappelijk Onderzoek (NWO). The authors thank the two anonymous reviewers for their careful reading and valuable suggestions that helped improve the exposition. This paper is dedicated to the memory of Prof. Luigi Ferretti, whose kindness and dedication inspired generations of students.
\bibliography{refs.bib}

\appendix
\section{Proofs of Auxiliary Lemmas}
\label{app1}
\begin{proof}[Proof of Lemma \ref{lemma:rewritehitprob}]
Consider any \(x\in A^c\setminus \SSS\). Then
\begin{align*}
    [Tf](x) &= \sum_{y\in \X} T(x,y)f(y)\\
    &=\sum_{y\in \X} T(x,y)\Bigl[\bigl(f|_{A^c\setminus S}\bigr)^{\uparrow_{\X}}(y) + \one_A(y)\Bigr]\\
    &=\Bigl[T\Bigl(\bigl(f|_{A^c\setminus \SSS}\bigr)^{\uparrow_{\X}}\Bigr)\Bigr](x) + [T\one_A](x)\\
    &=\bigl[T|_{A^c\setminus \SSS} (f|_{A^c\setminus \SSS})\bigr](x) + [T\one_A](x),
\end{align*}
where we used \(f(y)=1\) for \(y\in A\) and \(f(y)=0\) for \(y\in \SSS\). 
\end{proof}
\begin{proof}[Proof of Lemma \ref{lemma:412}]
By definition
\[
[T^n\one_A](x)=\sum_{z\in\X} T(x,z)[T^{n-1}\one_A](z).
\]
 All terms in the sum are nonnegative and by hypothesis the left-hand side is strictly greater than 0. Hence there exists \(y\in\X\) such that
\[
T(x,y)[T^{n-1}\one_A](y)>0,
\]
so in particular \(T(x,y)>0\) and \([T^{n-1}\one_A](y)>0\).

We now check that \(y\in A^c\setminus\C_T\). First, \(y\notin A\): if \(y\in A\) then \(T(x,y)>0\) would imply
\([T\one_A](x)\ge T(x,y)>0\), contradicting the minimality assumption that \([T^m\one_A](x)=0\) for every \(m<n\) (recall \(n\ge 2\)). Thus \(y\in A^c\).

Second, \(y\notin\C_T\): by definition any \(z\in\C_T\) satisfies \([T^m\one_A](z)=0\) for all \(m\in\N\), but we have \([T^{n-1}\one_A](y)>0\), so \(y\notin\C_T\).
Therefore \(y\in A^c\setminus\C_T\), \(T(x,y)>0\) and \([T^{n-1}\one_A](y)>0\), which proves the lemma.
\end{proof}
\begin{proof}[Proof of Lemma \ref{lemma:413}]
We first prove the monotonicity claim. Let \(M\coloneqq T|_{\SSS}\). If \(f\le g\) then for every \(x\in\SSS\)
\[
[ Mf](x)=\sum_{y\in\SSS} M(x,y) f(y)\le \sum_{y\in\SSS} M(x,y) g(y)=[ Mg](x),
\]
since each \(M(x,y)\ge0\). Thus \(Mf\le Mg\). Iterating this argument gives, by simple induction on \(k\), that \(M^k f\le M^k g\) for all \(k\in\N_0\).

We now prove the bounds for \(\mathbf{0}\) and \(\mathbf{1}\). First note that \(\mathbf{0}\le\mathbf{1}\), so by the monotonicity just proved we have
\[
M^k\mathbf{0}\le M^k\mathbf{1}\quad\text{for all }k\in\N_0.
\]
But \(M^k\mathbf{0}=\mathbf{0}\) for every \(k\). Hence \(\mathbf{0}\le M^k\mathbf{1}\).
For the upper bound, observe that for every \(x\in\SSS\)
\[
[ M\mathbf{1}](x)=\sum_{y\in\SSS} M(x,y)\le \sum_{y\in\X} T(x,y)=1.
\]
 Thus \(M\mathbf{1}\le\mathbf{1}\). Applying the monotonicity repeatedly yields
\[
M^{k}\mathbf{1}=M^{k-1}(M\mathbf{1})\le M^{k-1}\mathbf{1}\le\cdots\le M\mathbf{1}\le\mathbf{1},
\]
so \(M^k\mathbf{1}\le\mathbf{1}\) for every \(k\in\N_0\).

\end{proof}
\begin{proof}[Proof of Lemma \ref{lemma:414}]
Let \(M\coloneqq T|_{{A^c\setminus \C_T}}\) and fix \(x\in{A^c\setminus \C_T}\). Since \(x\notin\C_T\), there exists a (finite) minimal integer
\[
n_x\coloneqq\min\{n\ge 1:[T^n\one_A](x)>0\}.
\]
Consider any path of length \(n_x\), $x=x_0,\dots, x_{n_x}$, with $x_{n_x}\in A$.
Because \(m<n_x\) implies \([T^m\one_A](x)=0\), such a path cannot visit \(A\) before time \(n_x\).
Moreover, by the definition of \(\C_T\), a path that ever enters \(\C_T\) cannot reach \(A\) at any later time; hence every path contributing to \([T^{n_x}\one_A](x)\) stays in \({A^c\setminus \C_T}\) up to time \(n_x-1\) and then jumps from \({A^c\setminus \C_T}\) to \(A\) at time \(n_x\).

On the other hand, \([M^{n_x}{\mathbf{1}}](x)\) is the total probability of all length-\(n_x\) paths that \emph{remain in \({A^c\setminus \C_T}\) for all \(n_x\) steps}.
Therefore the two events are disjoint, and we obtain
\[
[M^{n_x}{\mathbf{1}}](x) \le 1 - [T^{n_x}\one_A](x) < 1.
\]
Thus, for every \(x\in{A^c\setminus \C_T}\), there exists \(n_x\ge 1\) such that \([M^{n_x}{\mathbf{1}}](x)<1\).

Now we use Lemma~\ref{lemma:413}: since \(M{\mathbf{1}}\le {\mathbf{1}}\), the map \(k\mapsto M^k{\mathbf{1}}\) is componentwise nonincreasing:
\(M^{k+1}{\mathbf{1}}=M^k(M{\mathbf{1}})\le M^k{\mathbf{1}}\) for all \(k\in\N_0\).
Let \(n\coloneqq\max_{x\in{A^c\setminus \C_T}} n_x\), which is finite because \({A^c\setminus \C_T}\) is finite.
Then for any \(k\ge n\) and any \(x\in{A^c\setminus \C_T}\)
\[
[M^{k}{\mathbf{1}}](x) \le [M^{n_x}{\mathbf{1}}](x) < 1,
\]
which is what we wanted to prove.
\end{proof}

\begin{proof}[Proof of Corollary \ref{cor:415}]
Let $x\in \SSS$. Then
\begin{align*}
    [(T |_{A^c\setminus \SSS})^k \mathbf{1}](x) &= \sum_{y\in A^c\setminus \SSS} (T |_{A^c\setminus \SSS})^k(x,y)\\
    &= \ \sum_{\mathclap{\substack{y\in  A^c\setminus \SSS \\ w_1,\dots,w_{k-1}\in \X }}} \quad T |_{A^c\setminus \SSS}(x,w_1)T |_{A^c\setminus \SSS}(w_1,w_2)\cdots T |_{A^c\setminus \SSS}(w_{k-1},y)\\
    &= \sum_{\substack{y\in  A^c\setminus \SSS \\ w_1,\dots,w_{k-1}\in A^c\setminus \SSS }}T (x,w_1)T (w_1,w_2)\cdots T (w_{k-1},y).
\end{align*}
Using $A^c\setminus \SSS \subseteq A^c\setminus \C_T$ we obtain
\begin{align*}
    [(T |_{A^c\setminus \SSS})^k \mathbf{1}](x) &=\sum_{\substack{y\in  A^c\setminus \SSS \\ w_1,\dots,w_{k-1}\in A^c\setminus \SSS }}T (x,w_1)T (w_1,w_2)\cdots T (w_{k-1},y)\\
    &\le \sum_{\substack{y\in  A^c\setminus \C_T \\ w_1,\dots,w_{k-1}\in A^c\setminus \C_T }}T (x,w_1)T (w_1,w_2)\cdots T (w_{k-1},y)\\
    &\qquad = [(T |_{A^c\setminus \C_T})^k \mathbf{1}](x)<1,
\end{align*}
which is what we wanted to prove.
\end{proof}
\begin{proof}[Proof of Corollary \ref{cor:416}]
Let $\SSS\subseteq\X$ satisfy $\C_T\subseteq \SSS\subset A^c$ and set
\[
M\coloneqq T|_{A^c\setminus\SSS}.
\]
By Corollary~\ref{cor:415} there exists $n\in\N$ such that
\[
[M^{n}\mathbf{1}](x)<1 \quad\text{for all }x\in A^c\setminus\SSS.
\]
Then, Lemma~\ref{lemma:413} yields $0\le M^{n}\mathbf{1}$.
Thus
\[
||M^{n}\mathbf{1}||_{\infty}<1.
\]
Since $M^n>0$ we have $||M^{n}||\le ||M^{n}\mathbf{1}||_{\infty}$ and the thesis follows.
\end{proof}

\begin{proof}[Proof of Lemma \ref{lemma:419}]
    Since $x\rightarrow^T y$ for some $y \in A$, we have $c\coloneqq p^T(x)>0$. Suppose $c=1$. Then, either $x\in A$ and we are done, or $x$ is surrounded by states with hitting probability 1, i.e. $p^T(w)=1$ for all $w\in \X$ such that $T(x,w)>0$. From the finiteness of $\X$, it follows that there exists the desired path.
    
If $0<c<1$, then define the following set of states:
\begin{equation*}
    C\coloneqq x\cup \{z\in \X:T(x,z)>0, \ p^T(z)=c\}.
\end{equation*}
This set is not closed: if it were the case, every state in $C$ could not reach $A$ and would have hitting probability equal to zero.
Using the fact that
\[
p^T(z)=\sum_{w\in \X} T(z, w) p^T(w),
\]
we obtain that either some neighbour $w\in \X$ of $z\in C$ has $p^T(w)>c$ or all neighbours of $z$ are in $C$. Therefore, there exist $m\in \N$ and a sequence $x=x_0,\dots, x_m$ such that $T(x_{i},x_{i+1})>0$ and $p^T(x_i)=c<p^T(x_m)$, for all $i=0,\dots, m-1$. Since $\X$ is finite, repeating this argument at most $N={\X}$ times gets the desired sequence of states.
\end{proof}

\end{document}